\documentclass[12pt]{amsart}

\usepackage{amsmath, amssymb, amscd, newlfont}

\newcommand{\End}{{\mathrm{End}}}
\newcommand{\Hom}{{\mathrm{Hom}}}

\newcommand{\Ann}{{\mathrm{Ann}}}
\newcommand{\Irr}{{\mathrm{Irr}}}

\newcommand{\calA}{{\mathcal{A}}}

\newcommand{\calO}{{\mathcal{O}}}

\newcommand{\calV}{{\mathcal{V}}}
\newcommand{\calW}{{\mathcal{W}}}

\usepackage[mathscr]{eucal}

\numberwithin{equation}{section}

\newtheorem{Lem}[equation]{Lemma}
\newtheorem{Def}[equation]{Definition}
\newtheorem{Thm}[equation] {Theorem}
\newtheorem{Cor}[equation]{Corollary}
\newtheorem{Rem}[equation]{Remark}

\title
  {On Representations of Reductive $p$--adic Groups over $\mathbb Q$--algebras}
\author{ Goran Mui\' c}

\address{ Department of Mathematics, Faculty of Science,
University of Zagreb,
Bijeni\v cka 30, 10000 Zagreb,
Croatia}
\email{gmuic@math.hr}
\thanks{The  author acknowledges Croatian Science Foundation grant no. 3628.}

\begin{document}
\maketitle
\begin{abstract}
  In this paper we study  certain category of smooth modules for reductive $p$--adic groups analogous to
  the usual smooth complex representations but with the field of complex numbers replaced by a $\mathbb Q$--algebra.
  We prove some fundamental results in these settings, and as an example we give a classification of admissible unramified irreducible
  representations proving by reduction to the complex case that if the space of $K$--invariants is finite dimensional in an irreducible
  smooth unramified representation that the representation is admissible. 
\end{abstract}

\section{Introduction}\label{agm}

In this paper we define and study  certain category of smooth modules for reductive $p$--adic groups analogous to
the usual smooth complex representations (\cite{Be, Be1, BZ, BZ1,cas}). Nowadays there is an active current research
in the field of complex representation theory as one can observe from the review articles \cite{tad1} and \cite{tad2}.
Representations in positive characteristic are also well understood thanks to the recent works of Henniart, Vign\' eras and others (see \cite{hv-2}).
But the representations of reductive $p$-adic groups on the vector spaces over extensions of $\mathbb Q$ such as number fields are not well--understood
beyond the study of fields of definition of complex representations \cite{robert}. In this paper we start to consider such problems.
On the example of a classification of unramified representations the reader will realize how reach and more interesting is this theory than the complex one
(but it is seems a lot more simpler than the case of positive characteristic \cite{hv-1}). It is based on the description of Satake isomorphism due to Gross \cite{gross}.

\vskip .2in
As with the approach  in  positive characteristic mentioned above, we use extensively Hecke algebra approach combined with the theory of semisimple algebras to reduce to the case
of algebraically closed field.  This not new, basic ideas can be found already in the book by Curtis and Reiner \cite{currei}.  The theory in positive characteristic is more involved and
it is based on a rather deep decomposition theorem (\cite{hv-2} Theorem I.1). In our case, we use just  use very basic theory of semismple rings (\cite{lang}, Chapter XVII) due to fact that we work
in characteristic zero.  We expect application
in the case of complex representations too but we leave it for another occasion.

\vskip .2in
In this paper rings are always associative commutative rings with $1\neq 0$ (as in \cite{mats}). Also homomorphism of rings always send $1$ onto $1$. A subring of a ring always contains the identity.
Ring modules are always unital i.e., $1$ acts as identity. We fix a non---Archiemedean local field $k$. Let  $G$ is a reductive $p$--adic group i.e., which by abuse of notation is 
is a group of $k$--points of a Zariski connected reductive group defined over $k$. As indicated at appropriate places, for some results in the paper we may assume that $G$ is just an $l$--group
(see \cite{BZ}) but for the the introduction with stick with the assumption that $G$ is a reductive $p$--adic group. I was informed by Casselman that new version of his classical book \cite{cas}
would contain extensive theory of parabolic induction and Jacquet modules for smooth representations with coefficients in the rings (see Definition \ref{agm-1}).

\vskip .2in 
We continue with expected form of the definition of modules that we consider. The following Definitions \ref{agm-1} and \ref{agm-2}  essentially taken from (\cite{Be}, 1.16, but see also
\cite{vig}, Chapter I):

\begin{Def}\label{agm-1}  Let $\mathcal A$ be ring. A $(\mathcal A, \ G)$--module is a 
$\mathcal A$--module  together with a homomorphism $G\longrightarrow GL_{\mathcal A}(V)$
such that  every element  in $V$ has an open stabilizer in $G$.  
\end{Def} 

Obviously, when $\mathcal A=\mathbb C$ we obtain usual smooth complex representation of $G$. 
More interesting example is when we use a center $\mathcal Z(G)$ of the category of smooth complex representations of $G$  (see \cite{Be}). The book
by Vign\' eras (\cite{vig}, Chapter I) contains many basis results for such modules. 

\vskip .2in 

Definition \ref{agm-1} implies that 
$$
V=\cup_L \ V^L \ \ \text{(the union ranges over all open compact subgroups of $G$)},
$$
and every $V^L$ is a  $\mathcal A$--module. 
 
\vskip .2in

When $\mathcal A=\mathbb C$, the definition below gives us  usual complex admissible representation  of $G$.

\begin{Def}\label{agm-2}  A $(\mathcal A, \ G)$--module $V$ is  
$\mathcal A$--admissible if $V^L$ is finitely generated $\mathcal A$--module for all open--compact subgroups $L\subset G$.
\end{Def}

\vskip .2in
We consider category 
$$
\mathcal C\left(\mathcal A, \ G\right)
$$
of all  $(\mathcal A, \ G)$--modules.
Obviously, $\mathcal C\left(\mathcal A, \ G\right)$ is an Abelian category. 

In what follows we assume that $\mathcal A$ is   $\mathbb Q$--algebra. Then, as expected, the functor $V\longmapsto V^L$
from the category  $\mathcal C\left(\mathcal A, \ G\right)$ into category of $\mathcal A$--modules is exact, for all open compact subgroups $L\subset G$ (see Lemma \ref{agm-31}).
An important consequence of the fact that we work with rings is the following fundamental result (see Lemma \ref{agm-3}):

\begin{Lem}\label{int-agm-3} Let $\mathfrak a\subset \mathcal A$ be an ideal of $\mathcal A$.
  Then, for any  $(\mathcal A, \ G)$--module $V$,
and  for any open compact subgroup $L\subset G$,  we have the following:
$$
 \left(\mathfrak a V\right)^L=\mathfrak a V^L.
 $$
 \end{Lem}

\vskip .2in
Since we work with the rings it is natural to consider the annhilator $\Ann_{\mathcal A}(V)$ in $\mathcal A$ of a
$(\mathcal A, \ G)$--module $V$. For irreducible but not $\mathcal A$--admissible modules $V$, the annhilator is just a prime ideal
(see Lemma \ref{agm-400} and the example after the proof of that lemma). But when $V$ is irreducible
and $\mathcal A$ the situation is much more manageable as can be seen from the theorem that we recall  below (see Theorem \ref{agm-4}).

\begin{Thm}\label{int--agm-4} Assume that $\mathcal A$ is a $\mathbb Q$--algebra.
  Then, we have the following:
  \begin{itemize}
  \item[(i)] For every irreducible  $\mathcal A$--admissible $(\mathcal A, \ G)$--module $V$,  the annhilator of $V$ is a
    maximal ideal. In fact, if we write $\mathfrak m= \Ann_{\mathcal A}(V)$, then  $V$ is an   irreducible $\mathcal A/\mathfrak m$--admissible $(\mathcal A/\mathfrak m, \ G)$--module.
\item[(ii)] Let $\Irr_{\mathfrak m}$ be  the set of equivalence classes of 
 irreducible $\mathcal A/\mathfrak m$--admissible $(\mathcal A/\mathfrak m, \ G)$--modules. Then,  the disjoint union 
 $$
 \cup_{\mathfrak m} \ \Irr_{\mathfrak m}  \ \ \text{($\mathfrak m$ ranges over maximal ideal of $\mathcal A$)}
 $$
 can be taken to be  the set of equivalence classes of  irreducible  $\mathcal A$--admissible 
 $(\mathcal A, \ G)$--modules.
\item[(iii)]  Assume that $\mathcal A$ is a  finitely generated $\mathbb C$--algebra. Let $\Irr(G)$ be the set of equivalence of
  complex irreducible admissible representations of $G$ (see \cite{BZ}). Let  $Max(\mathcal A)$ be the  set of all maximal ideals in $\mathcal A$. 
Then, the set $\Irr(G)\times Max(\mathcal A)$ parametrizes irreducible $\mathcal A$--admissible  $(\mathcal A, \ G)$--modules. 
\end{itemize}
\end{Thm}

\vskip .2in
Lemma \ref{int-agm-3} recalled above is of the fundamental importance in the proof of this theorem. Maintaining the notation of the theorem,
the identity action of $G$ on $\mathcal A/\mathfrak m$ is an example of irreducible $\mathcal A/\mathfrak m$--admissible
$(\mathcal A/\mathfrak m, \ G)$--module. We call it the trivial representation. Therefore, $\Irr_{\mathfrak m}$ is always non--empty.
When $G$ is a reductive $p$--adic group, we will prove the existence of other more complicated representation. But in the present generality, $G$ could be the trivial group, and
we can not do better.

\vskip .2in

Section \ref{eir} discusses the existence of irreducible
$(\mathcal A, \ G)$--modules  via  Hecke algebra  adapted from the classical complex case \cite{BZ}. (See also \cite{vig}, Chapter I, or \cite{cas}.)
Let  $\mathcal H(G, \ \mathcal A)$ be the Hecke algebra of $\mathcal A$--valued locally constant
and compactly supported functions on $G$ and $\mathcal H(G, \ L, \  \mathcal A)$ its sublagebra of all $L$--biinvariant functions in   $\mathcal H(G, \ \mathcal A)$ for $L\subset G$ open compact.
Usual relation between non--degenerate   $\mathcal H(G, \ \mathcal A)$--modules and smooth  $(G, \ \mathcal A)$--modules is valid as well as usual results for
irreducible   $(G, \ \mathcal A)$--module regarding irreducibility of $V^L$.
The main result of Section \ref{eir} is Theorem \ref{eir-9} in which we give very explicit construction of irreducible  $(G, \ \mathcal A)$--module $V$ from the known
irreducible module $V^L$ for $\mathcal H(G, \ L, \  \mathcal A)$. This is an improvement over the classical treatment in   (\cite{BZ}, Proposition 2.10 c)) and it is need for many
results that follow in this paper such as the description of ring of endomorphisms in Theorem \ref{aeir-1} which is the main result of Section \ref{aeir}, as well as the following 
fundamental result which is the main result of Section \ref{aaeir} (see Theorem \ref{aaeir-1}):

\begin{Thm}\label{int-aaeir-1} Assume that $\mathcal A$ is a field  (and then extension of $\mathbb Q$). Let $L\subset G$ be an open compact subgroup.
  Let $V$ be an irreducible $(\mathcal A, \ G)$--module such that $V^L\neq 0$ and  $\mathcal A$--finite dimensional (i.e., $V^L$ is an $\mathcal A$--admissible 
  irreducible $\mathcal H\left(G, \ L, \ \mathcal A\right)$--module). Then,  for any field extension $\mathcal A\subset \mathcal B$, there exists   irreducible  $(\mathcal B, \ G)$--modules
  $V_1, \ldots, V_t$ such that the following holds:
  \begin{itemize}
  \item[(i)] $V^L_i\neq 0$ for all  $1\le i\le t$. 
  \item[(ii)] $V^L_i$ are  $\mathcal B$--admissible  irreducible $\mathcal H\left(G, \ L, \ \mathcal B\right)$--modules.
  \item[(iii)] $V_{\mathcal B}\overset{def}{=}\mathcal B \otimes_{\mathcal A} V \simeq V_1\oplus \cdots \oplus V_t$ as $(\mathcal B, \ G)$--modules.
    \end{itemize}
\end{Thm}

\vskip .2in
We warn the reader that we do not assume that $V$ is $\mathcal A$--admissible but that $V^L\neq 0$ is $\mathcal A$--admissible.  
On the level of $L$--invariants, the decomposition in (iii) is contained in Lemma \ref{aaeir-2} and it is based on some very simple  facts from the theory of
semi--simple rings  (\cite{lang}, Chapter XVII). A more complicated case of positive characteristic  require more elaborated tools (\cite{hv-1}, Theorem I.1).

\vskip .2in
We warn the reader that because of Theorem \ref{int--agm-4}, the assumption that $\mathcal A$ is a field is expected. Theorem \ref{int-aaeir-1} has many applications. They are contained in
Section \ref{ebf}. We recall just the following one (see Corollary \ref{ebf-0}):

\begin{Cor}\label{int--ebf-0} Assume that $\mathcal A$ is any subfield of $\mathbb C$.
   Let $L\subset G$ be an open compact subgroup.
  Let $V$ be an irreducible $(\mathcal A, \ G)$--module such that $V^L\neq 0$ and  $\mathcal A$--finite dimensional. Then, $V$ is $\mathcal A$--admissible (see Definition \ref{agm-2}).
\end{Cor}

\vskip .2in
This is proved reducing to the well--known result in the complex case via Theorem \ref{int-aaeir-1}. We remind the reader that
by a result of Jacquet (\cite{renard}, Theorem VI.2.2), every irreducible $(\mathbb C, \ G)$--irreducible module is $\mathbb C$--admissible.
But in the generality that we consider we are not sure that every irreducible $(\mathcal A, \ G)$--module is  $\mathcal A$--admissible without assumptions  stated in Corollary \ref{int--ebf-0}.
In the present form  Corollary \ref{int--ebf-0} is quite useful since it fundamentally contributes to the construction of unramified irreducible representations (see Theorem \ref{cuir-4} in
Section \ref{cuir}):

\begin{Thm}\label{int--cuir-4} Let $k$ be a non--Archimedean local field. Let $\mathcal O\subset k$ be its ring of integers, and
  let $\varpi$ be  a generator of the maximal ideal in $\calO$. Let $q$ be the number of elements in the residue field $\mathcal O/ \varpi\mathcal O$.
  Assume that is $G$  is a  $k$--split Zariski connected reductive group. Let $A$ be its maximal $k$--split torus, and $W$ the corresponding group.
  We write $\hat{A}$ for the complex torus dual to $A$. Let $W$ be the Weyl group of $A$ in $G$. The orbit space
  $$
 X\overset{def}{=}\hat{A}/W
 $$
is an affine variety defined over $\mathbb Q$.  Let $K=G(\mathcal O)$ be its  hyperspecial
maximal compact subgroup of $G$.  We normalize a Haar measure on $G$  such that $\int_{K}dg=1$  (see Section \ref{eir}). Let $\overline{\mathbb Q}$ be the algebraic closure of $\mathbb Q$ inside $\mathbb C$.
Let $\mathcal A$ be any  subfield of $\overline{\mathbb Q}$ if $G$ is simply--connected, or a extension  of $\mathbb Q(q^{1/2})$ in $\overline{\mathbb Q}$ otherwise.  We define the (commutative) Hecke algebra
$\mathcal H\left(G, \  K, \ \mathcal A\right)$ with respect to above fixed Haar measures. Then, we have the following:

\begin{itemize}
\item[(i)] (Satake isomorphims over subfields of $\overline{\mathbb Q}$)
  Maximal ideals in $\mathcal H\left(G, \  K, \ \mathcal A\right)$ are parame\--trized by points in $X(\overline{\mathbb Q})$ such that  points in $X(\overline{\mathbb Q})$
  give the same maximal ideal  if and only if they are  $Gal(\overline{\mathbb Q}/\mathcal A)$--conjugate: for $x\in  X(\overline{\mathbb Q})$, we denote by
  $\mathfrak m_{x, \mathcal A}$ the corresponding maximal ideal.   The corresponding
  quotient $\mathcal H\left(G, \  K, \ \mathcal A\right)/\mathfrak m_{x, \mathcal A}$ is  denoted by $F(x, \mathcal A)$. It is a finite (field)  extension of $\mathcal A$, and it also naturally
  irreducible $\mathcal A$--admissible $\mathcal H\left(G, \  K, \ \mathcal A\right)$--module. The map $Gal(\overline{\mathbb Q}/\mathcal A).x\longmapsto  F(x, \mathcal A)$  is a bijection 
  between $Gal(\overline{\mathbb Q}/\mathcal A)$--orbits in  $X(\overline{\mathbb Q})$, and the set of equivalence classes of irreducible $\mathcal A$--admissible irreducible
  $\mathcal H\left(G, \  K, \ \mathcal A\right)$--modules.

\item[(ii)] For each $x\in X(\overline{\mathbb Q})$, the $(\mathcal A, \  G)$--module (see Theorem \ref{eir-9} for the notation)
  $$
  \mathcal V(x, \mathcal A)\overset{def}{=}\mathcal V(\mathfrak m_x, K)
  $$
  is an irreducible and {\bf $\mathcal A$--admissible} $(\mathcal A, \ G)$--module. We have
  $$
  \mathcal V^K(x, \mathcal A)\simeq \mathcal B_{x, \mathcal A}
  $$ as $\mathcal H\left(G, \  K, \ \mathcal A\right)$--modules, 
  and
  $$
  \End_{(\mathcal A, \ G)} \left( \mathcal V(x, \mathcal A)\right)\simeq F(x, \mathcal A).
$$

\item[(iii)] $\mathcal V(x, \mathcal A)$ is absolutely irreducible (see Corollary \ref{eib-99} for the standard definition of absolute irreducibility) if and only if $x\in X(\mathcal A)$.

\item[(iv)] Let $x\in X(\overline{\mathbb Q})$. Then, for any Galois extension  $\mathcal A\subset \mathcal B$ which contains $F(x, \mathcal A)$,
  $\mathcal V(x, \mathcal B)$ is absolutely irreducible. Moreover, there exist $t=\dim_{\mathcal A} \ F(x, \mathcal A)$ 
  mutually different elements (among them $x$) in $Gal(\overline{\mathbb Q}/\mathcal B).x$, say $x=y_1, y_2, \ldots, y_t$, 
  such that  we have the following:
  $$
  \left( \mathcal V(x, \mathcal A)\right)_{\mathcal B}= \mathcal B\otimes_{\mathcal A} \  \mathcal V(x, \mathcal A)\simeq \mathcal V(x, \mathcal B)
\oplus   \mathcal V(y_2, \mathcal B)
\oplus \cdots \oplus  \mathcal V(y_t, \mathcal B).
$$
Furthermore,   $V(x, \mathcal B),  \mathcal V(y_2, \mathcal B),  \ldots,  \mathcal V(y_t, \mathcal B)$ are mutually non--isomorphic $(\mathcal B, \ G)$--modules.

\item[(v)] (Classification of unramified admissible representations  subfields of $\overline{\mathbb Q}$) The map
  $$
  Gal(\overline{\mathbb Q}/\mathcal A).x\longmapsto  \mathcal V(x, \mathcal A)
  $$
  is a bijection between $Gal(\overline{\mathbb Q}/\mathcal A)$--orbits in  $X(\overline{\mathbb Q})$, and the set of equivalence classes of unramified $\mathcal A$--admissible
  irreducible  $(\mathcal A, \ G)$--modules. 
\end{itemize}
\end{Thm}

\vskip .2in
Besides above mentioned result, the key point is the description of Satake isomorphism \cite{cartier} over $\mathbb Z$ due to Gross \cite{gross} and a technical lemma
about affine varieties proved in the Appendix (see Lemma {acuir-4} in Section \ref{acuir}).

\vskip .2in
The first ideas about the content  of the paper were realized  while the author visited the Hong Kong University
of Science and Technology in January  of 2018. The author would like to thank A. Moy and the
Hong Kong University of Science and Technology for their hospitality. I would like to thank Marko Tadi\' c for showing me the
reference \cite{robert}.  The discussions with Marie--France Vign\' eras and William Casselman were useful in the process of revision of the manuscript.
Marie--France Vign\' eras kindly provided  me with the references \cite{hv-1}, \cite{hv-2} and \cite{vig}.

\section{Basic properties of $\left(\mathcal A, \ G\right)$--modules} 

In this section we assume that  $G$ is a $l$--group (\cite{BZ}, 1.1): this means that $G$ is Hausdorff and  there is a
fundamental system of neighborhoods of the unit element consisting of open compact subgroups.  We always assume that $\mathcal A$ is a $\mathbb Q$--algebra.
In this section we prove basic properties of $\left(\mathcal A, \ G\right)$--modules.

\vskip .2in 

We start with the following result:

\begin{Lem}\label{agm-31}  The functor $V\longmapsto V^L$
  from the category  $\mathcal C\left(\mathcal A, \ G\right)$ into category of $\mathcal A$--modules is exact.
\end{Lem}
\begin{proof} It is enough to show that if $V_1\longrightarrow V_2 \longrightarrow V_3$ is an exact sequence in  $\mathcal C_{adm}\left(\mathcal A, \ G\right)$, then
  $V^L_1\longrightarrow V^L_2 \longrightarrow V^L_3$ is also exact. It is obvious that the image of $V_1^L$ is contained in the kernel of  $V^L_2 \longrightarrow V^L_3$.
  Conversely, let $v$ be an element in the kernel of  $V^L_2 \longrightarrow V^L_3$. Then, there exists $w\in V_1$ which image is $v$ under the map  $V_1\longrightarrow V_2$.
  Let $L'\subset L$ be an open compact subgroup such that $w\in V_1^{L'}$ and $v\in V^{L'}_2$. Let
 $$
  w_0 =\frac{1}{\# \left(L/L'\right)} \left(\sum_{\gamma\in L/L'} \gamma.w\right)
  $$
  Then, $w_0\in V^{L}_1$, and $v$ is image of $w_0$ under the map  $V_1\longrightarrow V_2$ since $v$ is $L$--stable.
  \end{proof}

\vskip .2in

\begin{Lem}\label{agm-30} Assume that $\mathcal A$ is a Noetherian ring. Then, 
  $\mathcal C_{adm}\left(\mathcal A, \ G\right)$ is an Abelian category.
\end{Lem}
\begin{proof} Let $L\subset G$ be an open--compact subgroup. Let $V$ be an object in $\mathcal C_{adm}\left(\mathcal A, \ G\right)$. Then, by definition
  $V^L$ is finitely generated $\mathcal A$--module. If $W\subset V$ is a submodule, then $W^L\subset V^L$. Hence, $W^L$ is finitely generated $\mathcal A$--module
  since $\mathcal A$ is a Noetherian ring. Next, if $U$ is a quotient module of $V$. Then, $U^L$ is a quotient module of $V^L$. Now, we apply Lemma \ref{agm-31}
  to prove that $U^L$ is finitely generated $\mathcal A$--module.  This shows that submodules and quotients belongs to  $\mathcal C_{adm}\left(\mathcal A, \ G\right)$. This implies that
  category  $\mathcal C_{adm}\left(\mathcal A, \ G\right)$ is Abelian. 
  \end{proof}

\vskip .2in 
The following lemma is one of the key technical results:

\begin{Lem}\label{agm-3} Let $\mathfrak a\subset \mathcal A$ be an ideal of $\mathcal A$.
  Then, for any  $(\mathcal A, \ G)$--module $V$,
and  for any open compact subgroup $L\subset G$,  we have the following:
$$
 \left(\mathfrak a V\right)^L=\mathfrak a V^L.
 $$
 \end{Lem}
\begin{proof} Obviously, we have

$$
\mathfrak a V^L\subset  \left(\mathfrak a V\right)^L,
$$
for any open--compact subgroup $L$.

Let $v\in  \left(\mathfrak a V\right)^L$. Then, there exists $v_1, \ldots, v_l\in V$, $a_1, \ldots, a_l\in \mathfrak a$ such that 
$$
v=\sum_{i=1}^l a_iv_i.
$$
We select $L'\subset L$ an open compact subgroup such that $v_1\ldots, v_l\in V^{L'}$. Then
$$
\# \left(L/L'\right) \cdot v =\sum_{i=1}^l a_i \left(\sum_{\gamma\in L/L'} \gamma.v_i\right).
$$
Obviously, we have
$$
\sum_{\gamma\in L/L'} \gamma.v_i \in V^L.
$$
Thus, we have
$$
\# \left(L/L'\right) \cdot v \in \mathfrak a V^L.
$$
\end{proof}

\vskip .2in
Now, we consider the ring of all endomorphims of $\End_{(\mathcal A, \ G)}\left(V\right)$ of  an irreducible  $(\mathcal A, \ G)$--module $V$.
See also Theorem \ref{aeir-1} where we relate to the Hecke algebras.  We remark that when $G$ is countable at infinity, and $\mathcal A=\mathbb C$,
this ring is always isomorphic to $\mathbb C$ (see \cite{BZ}, Proposition 2.11).  In general, the situation is more interesting.

\begin{Lem} \label{agm-400} Let  $V$ be an irreducible  $(\mathcal A, \ G)$--module.
  Then, the annhilator  of $V$, denoted by  $\Ann_{\mathcal A}(V)$,   in $\mathcal A$ is a  prime ideal.
 Moreover, if we let $\mathfrak p= \Ann_{\mathcal A}(V)$, then the module $V$ extends naturally to an irreducible representation of
  $(k(\mathfrak p), \ G)$, where $k(\mathfrak p)$ is the field of fractions of $A/\mathfrak p$.
  The ring of all endomorphisms $\End_{(\mathcal A, \ G)}\left(V\right)$ is a divison algebra naturally isomorphic to
 $\End_{(k(\mathfrak p), \ G)}\left(V\right)$, and therefore central over $k(\mathfrak p)$. 
\end{Lem}
\begin{proof}  By definition of a prime ideal, we need to show that $ab\in Ann_{\mathcal A}(V)$ implies $a\in Ann_{\mathcal A}(V)$ or $b\in Ann_{\mathcal A}(V)$.
  Indeed, if $b\not\in Ann_{\mathcal A}(V)$, then $bV$ is a non--zero  $(\mathcal A, \ G)$--submodule $V$. Hence, $bV=V$ because $V$ is irreducible.  Hence,
  $$
  aV=a\left(bV\right)=abV=0,
  $$
  since $ab\in Ann_{\mathcal A}(V)$. This implies $a\in Ann_{\mathcal A}(V)$.

  By Schur's lemma, $\End_{(\mathcal A, \ G)}\left(V\right)$ is a divison algebra. Obviously, $\mathcal A/\mathfrak p$ embeddes into the center  of
$\End_{(\mathcal A, \ G)}\left(V\right)$. The center is a field. Thefore, $k(\mathfrak p)$ embeddes into the center. Now, $V$ can be regarded as
as a $(k(\mathfrak p), \ G)$--module. It is obviously irreducible since $V$ was originally irreducible $(\mathcal A, \ G)$--module. Next, it is clear that
$$
\End_{(k(\mathfrak p), \ G)}\left(V\right)\subset \End_{(\mathcal A, \ G)}\left(V\right).
$$
Finally, since
$k(\mathfrak p)$ belongs to the center of $\End_{(\mathcal A, \ G)}\left(V\right)$, we have
$$
\End_{(\mathcal A, \ G)}\left(V\right)\subset \End_{(k(\mathfrak p), \ G)}\left(V\right).
$$
\end{proof}

\vskip .2in
Here is an example for Lemma \ref{agm-400}. Consider the ring  of polynomials $\mathbb Q[T]$ over $\mathbb Q$. Then, we let $\mathcal A$ to be the localization of  $\mathbb Q[T]$
at the prime ideal generated by $T$. Let $\mathcal K$ be the field of fractions of $\mathbb Q[T]$ and of $\mathcal A$. Then, $\mathcal A$ is a $\mathbb Q$--algebra and a local ring with a unique maximal ideal,
say $\mathfrak m$,  the one generated by $T$.  We let $G=\mathcal K^\times$ and equip  it with a discrete topology. In this way, we obtain an
$l$--group.  Let $V=\mathcal K$.  Then,  $V$ is in an obvious way an irreducible  $(\mathcal A, \ G)$--module. Its annhilator is a $\{0\}$ which is a prime ideal in $\mathcal A$. We remark that $V$ is not
$\mathcal A$--admissible since $\mathcal K$ is not finitely generated over $\mathcal A$. We remark also $\mathfrak m V=V$, and 
$\End_{(\mathcal A, \ G)}\left(V\right)=\mathcal K$.  Finally, we remark that $G$ is countable at infinity since it is a countable set.

\vskip .2in
The following theorem gives further description of irreducible $\mathcal A$--admissible modules and an improvement over Lemma \ref{agm-400}:

\begin{Thm}\label{agm-4} Assume that $\mathcal A$ is a $\mathbb Q$--algebra, and $G$ an $l$--group.
  Then, we have the following:
  \begin{itemize}
  \item[(i)] For every irreducible  $\mathcal A$--admissible $(\mathcal A, \ G)$--module $V$,  the annhilator of $V$ is a
    maximal ideal. In fact, if we write $\mathfrak m= \Ann_{\mathcal A}(V)$, then  $V$ is an   irreducible $\mathcal A/\mathfrak m$--admissible $(\mathcal A/\mathfrak m, \ G)$--module.
\item[(ii)] Let $\Irr_{\mathfrak m}$ be  the set of equivalence classes of 
 irreducible $\mathcal A/\mathfrak m$--admissible $(\mathcal A/\mathfrak m, \ G)$--modules. Then,  the disjoint union 
 $$
 \cup_{\mathfrak m} \ \Irr_{\mathfrak m}  \ \ \text{($\mathfrak m$ ranges over maximal ideal of $\mathcal A$)}
 $$
 can be taken to be  the set of equivalence classes of  irreducible  $\mathcal A$--admissible 
 $(\mathcal A, \ G)$--modules.
\item[(iii)]  Assume that $\mathcal A$ is a  finitely generated $\mathbb C$--algebra. Let $\Irr(G)$ be the set of equivalence of
  complex irreducible admissible representations of $G$ (see \cite{BZ}). Let  $Max(\mathcal A)$ be the  set of all maximal ideals in $\mathcal A$. 
Then, the set $\Irr(G)\times Max(\mathcal A)$ parametrizes irreducible $\mathcal A$--admissible  $(\mathcal A, \ G)$--modules. 
\end{itemize}
 \end{Thm}
\begin{proof} We prove (i). Since $V$ is irreducible, for each maximal ideal $\mathfrak m\subset \mathcal A$, we have $\mathfrak mV=0$ or $\mathfrak mV=V$.
  Assume that $\mathfrak m V=V$ for all $\mathfrak m$. Then, for an open compact subgroup $L\subset G$, applying Lemma \ref{agm-3}, we must have
$$
V^L= \left(\mathfrak m V\right)^L=\mathfrak m V^L,
 $$
for all $\mathfrak m$. Then, because of Lemma \ref{pav-5}, we must have $V^L=0$. Since $L$ is arbitrary, we obtain $V=0$.  This is a contradiction. Thus, there exists at least one maximal ideal 
$\mathfrak m$ such that $\mathfrak m V=0$. Then, $\mathfrak m\subset Ann_{\mathcal A}(V)$. Hence, 
$$
Ann_{\mathcal A}(V)=\mathfrak m.
$$

It is is obvious that (ii) follows from (i) at once. Finally for (iii), we remark that by Nullstellensatz  $A/\mathfrak m=\mathbb C$ for each $\mathfrak m\in Max(\mathcal A)$.
Hence, (iii) is an obvious consequence of (ii). 
\end{proof}

\vskip .2in
Maintaining the notation of the theorem, the identity action of $G$ on $\mathcal A/\mathfrak m$ is an example of irreducible $\mathcal A/\mathfrak m$--admissible
$(\mathcal A/\mathfrak m, \ G)$--module. We call it the trivial representation. Therefore, $\Irr_{\mathfrak m}$ is always non--empty.
When $G$ is a reductive $p$--adic group, we will prove the existence of other more complicated representation. But in the present generality, $G$ could be the trivial group, and
we can not do better. Section \ref{eir} discusses the existence of irreducible
$(\mathcal A, \ G)$--modules  via  Hecke algebra  adapted from the classical complex case \cite{BZ}.

\vskip .2in 
The following general result   follows from (\cite{mats}, Chapter 4, Theorems 4.6., 4.8) and it is needed in the proof of Theorem \ref{agm-4}:

\begin{Lem} \label{pav-5} Let $V$ be a finitely generated unital module over a commutative ring $R$ with identity. 
 Then, if  $\mathfrak m V=V$ for all maximal ideals  $\mathfrak m\subset R$,  then $V=0$.
 \end{Lem}
\begin{proof} We include the proof for the sake of completeness. Let $V_{\mathfrak m}$ be the localization of $V$ at $\mathfrak m$. Then,
  by the assumption of the lemma and Nakayama's lemma, $V_{\mathfrak m}=0$.

  Let $v\in V$.   Then, by above observation, there exists
$$
s_{v, \mathfrak m}\in  R-\mathfrak m
$$
such that
$$
s_{v, \mathfrak m}.v=0 \ \ \text{in $V$,}
$$
for all maximal ideals $\mathfrak m$.

The collection of all $s_{v, \mathfrak m}$, where $\mathfrak m$ ranges over all maximal ideals of $R$, generates an ideal, say $I$, that is not contained in any
$\mathfrak m$. But, then
$$
I=R.
$$
Thus, there exists $\mathfrak m_1, \ldots, \mathfrak m_k$, and $r_1, \ldots, r_l\in R$ such that
$$
1_R=\sum_{i=1}^lr_is_{v, \mathfrak m_i}.
$$

Then, we have
$$
v=1_R.v=\sum_{i=1}^lr_is_{v, \mathfrak m}.v=0.
$$
This proves $V=0$.
       \end{proof}

\vskip .2in
Let $\mathcal A\subset \mathcal B$ be an extension of rings.  Then, for  $(\mathcal A, \ G)$--module $V$ we can consider
$(\mathcal B, \ G)$--module defined as follows:
$$
V_{\mathcal B}\overset{def}{=} V_{\mathcal B/\mathcal A}\overset{def}{=} \mathcal B \otimes_{\mathcal A} V.
$$

\begin{Lem} \label{agm-7}  Assume that $\mathcal A$ is a $\mathbb Q$--algebra. Then, 
under the above assumptions, we have the following:
\begin{itemize}
\item[(i)]  For each open compact subgroup $L\subset G$, we have the following:
$$
V^L_{\mathfrak B} =\mathcal B\otimes_{\mathcal A} V^L.
$$
\item[(ii)] The $V_{\mathcal B}$ is $\mathcal B$--admissible whenever $V$ is $\mathcal A$--admissible.

\item[(iii)] The assignment  $V\longmapsto V_{\mathcal B}$  can be regarded as a functor
  $\mathcal C\left(\mathcal A, \ G\right)\longrightarrow \mathcal C\left(\mathcal B, \ G\right)$ and as a functor
 $\mathcal C_{adm}\left(\mathcal A, \ G\right)\longrightarrow \mathcal C_{adm}\left(\mathcal B, \ G\right)$.
\end{itemize}
\end{Lem}
\begin{proof} (i) has the proof similar to the proof of Lemma \ref{agm-3}. (ii) follows from (i). Finally,  the first functor in (iii) is obvious.
  The second one is well--defined because of (ii). 
\end{proof}

\vskip .2in
Let  $V$ be a $(\mathcal A, \ G)$--module.
Let $\mathfrak p\subset \mathcal A$ be a prime ideal, and $\mathcal A_{\mathfrak p}$ be the 
localization of $\mathcal A$ at $\mathfrak p$. Then, by the standard commutative algebra, $V_{\mathcal A_{\mathfrak p}}$ is the localization of $V$ at $\mathfrak p$ considered
as a $\mathcal A$--module. We denote it by $V_{\mathfrak p}$.
\vskip .2in

\begin{Thm} \label{agm-8} Assume that $\mathcal A$ is a $\mathbb Q$--algebra, and $G$ an $l$--group. Let  $V$ be an irreducible  $\mathcal A$--admissible $(\mathcal A, \ G)$--module.
  Then, for a  prime ideal $\mathfrak p\subset \mathcal A$,  we have the following:
    $$
    V_{\mathfrak p}=\begin{cases}\text{is $\mathcal A_{\mathfrak p}$--admissible irreducible
      $(\mathcal A_{\mathfrak p}, \ G)$--module,}  \ \ \text{if $\mathfrak p= Ann_{\mathcal A}(V)$,}\\
    0,  \ \ \text{if $\mathfrak p\neq Ann_{\mathcal A}(V)$.}\\
    \end{cases}
    $$
    Moreover, if $\mathfrak p= Ann_{\mathcal A}(V)$, then
    $$
    Ann_{\mathcal A_{\mathfrak p}}(V_{\mathfrak p})= \mathfrak m_{\mathfrak p},
    $$ 
    where the right--hand side is the  the localization of $\mathfrak p$.  Using the canonical isomorphism of localizations
    $\mathcal A/\mathfrak p\simeq  \mathcal A_{\mathfrak p}/\mathfrak m_{\mathfrak p}$,  $V_{\mathfrak p}$ is isomorphic to
    $V$ as  a $(\mathcal A/ \mathfrak p, \ G)$--module.
\end{Thm}
\begin{proof} We recall that  $Ann_{\mathcal A}(V)$ is a maximal ideal. Therefore, if  $\mathfrak p\neq Ann_{\mathcal A}(V)$ is a prime ideal,
  then $Ann_{\mathcal A}(V)-\mathfrak p\neq \emptyset$. Select $x\in  Ann_{\mathcal A}(V)-\mathfrak p$. Then $x/1\in \mathcal A_{\mathfrak p}$ is invertible and it acts as zero
  on $V_{\mathfrak p}$ . Thus, $V_{\mathfrak p}$  is zero. 

  Assume $\mathfrak p=Ann_{\mathcal A}(V)$. Then, the maximal ideal $\mathfrak m_{\mathfrak p}$, obtained by the localization of $\mathfrak p$, obviously annhilates
  $V_{\mathfrak p}$. None of the other elements 
in   $\mathcal A_{\mathfrak p}$ can kill $V_{\mathfrak p}$ since by the properties of the localization and 
irreducibility of $V$ would exist an $s\in \mathcal A-\mathfrak p$ which kills $V$ which is not possible. This proves
$Ann_{\mathcal A_{\mathfrak p}}(V)= \mathfrak m_{\mathfrak p}$. 

Next, we may regard $V_{\mathfrak p}$ as $(\mathcal A/ \mathfrak p, \ G)$--module. Hence, 
the argument similar to the one used in the computation of the annhilator above shows that
$V\longrightarrow V_{\mathfrak p}$ , given by $v\longmapsto 1\otimes v$ is injective map of $(\mathcal A/ \mathfrak p, \ G)$--modules.
Since, the usual properties of localization imply
$$
\mathcal A/\mathfrak p\simeq  \mathcal A_{\mathfrak p}/\mathfrak m_{\mathfrak p},
$$
Hence, the map is an isomorphism of  $(\mathcal A/ \mathfrak p, \ G)$--modules.  Hence, $V_{\mathfrak p}$ is  irreducible
$(\mathcal A_{\mathfrak p}, \ G)$--module. It is $\mathcal A_{\mathfrak p}$--admissible by Lemma \ref{agm-7} (ii).
\end{proof}

\section{Existence of irreducible representations}\label{eir}

In this section we assume that  $\mathcal A$ is a $\mathbb Q$--algebra, and $G$ an $l$--group. The goal of this section is to discuss existence of irreducible
$(\mathcal A, \ G)$--modules.  As it may be expected, we use Hecke algebra  adapted from the classical complex case \cite{BZ} but there are some improvement of the
classical complex case. The main result of this section is  Theorem \ref{eir-9}. We remark that the basic idea of the present approach
to the construction of Hecke algebra over $\mathcal A$ was already well--known  (see \cite{NT}, 2.2, where the case of profinite groups).

\vskip .2in

Let $L\subset G$ be an open compact subgroup. Let $\mathcal A$ be a $\mathbb Q$--algebra. We consider
the space $\mathcal H\left(G, L, \mathcal A\right)$ of all functions $f:G \longrightarrow \mathcal A$ which are $L$--biinvariant
and have compact support i.e., they are supported on finitely many double cosets $LxL$, where $x\in G$.  If $1_T$ denotes the characteristic function of a subset $T\subset G$, then
every function $f\in \mathcal H\left(G, L, \mathcal A\right)$ can be written uniquely in the form:
$$
f=\sum_{x\in L\backslash G / L} a_x \cdot 1_{LxL}, \ \ \text{where $a_x\in \mathcal A$, equal to zero for all but finitely many $x$.}
$$
The Hecke algebra  $\mathcal H\left(G, \mathcal A\right)$ with coefficients in $A$ is just the union of all  $\mathcal H\left(G, L, \mathcal A\right)$ when $L$ ranges over all
open compact subgroups of $G$. 

When $\mathcal A=\mathbb C$, we obtain usual Hecke algebras (\cite{cas}, \cite{BZ})  The product is given by the convolution
$$
f\star g(x)= \int_G f(xy^{-1}) g(y) dy.
$$
We recall that $\mathcal H\left(G, L, \mathbb C\right)$ is associative $\mathbb C$--algebra with identity $1_L/vol(L)$. It is a subalgebra of $\mathcal H\left(G, \mathbb C\right)$ for all
$L$.  As it is easy to see and also can be seen by inspecting the construction of Haar measure on $G$ (see the proof of Proposition 1.18 in \cite{BZ}), we see that if we select an open compact
subgroup and require that its volume is equal to one (a rational number!), then all volumes of all open compact subgroups are rational. Moreover, above defined convolution $\star$ makes
$\mathcal H\left(G, \mathbb Q\right)$ into an associative $\mathbb Q$--algebra (in general without identity), and  $\mathcal H\left(G, L, \mathbb Q\right)$ an
associative $\mathbb Q$--algebra with identity $1_L/vol(L)$. Let us explain why $\mathcal H\left(G, \mathbb Q\right)$ is closed under convolution. The reader can easily show that this boils down
to show that $1_{xL}\star 1_{yL}\in \mathcal H\left(G, \mathbb Q\right)$ for all $x, y\in G$, and open compact subgroups $L\subset G$. Indeed, we have the following:

\begin{equation}\label{eir-1}
  \begin{aligned}
  1_{xL}\star 1_{yL}(z)&=\int_G 1_{xL}(zt^{-1}) 1_{yL}(t) dt\\
  & =\int_{yL}  1_{xL}(zt^{-1})dt=vol\left(Lx^{-1}z\cap yL\right)\\
  &= M(x, y, z) \cdot vol \left(L\cap yLy^{-1} \right) \in \mathbb Q,
  \end{aligned}
\end{equation}
where $M(x, y, z)$ is the number of right cosets of the open compact subgroup $L\cap yLy^{-1}$ in which is decomposed  $Lx^{-1}z\cap yL$. We remark that $Lx^{-1}z\cap yL\neq \emptyset$ is equivalent to
$zL=xl_1yL$ for some $l_1\in L$ determined uniquely modulo left coset
$l'_1\left(L\cap yLy^{-1}\right) $. Also, we have the following:
$$
Lx^{-1}z\cap yL=Ll_1y\cap yL=Ly\cap yL= \left(L\cap yLy^{-1} \right) \cdot y
$$
This implies that $M(x, y, z)=1$ whenever $Lx^{-1}z\cap yL\neq \emptyset$.

An explicit computation using defining integral shows
that $1_{xL}\star 1_{yL}$ is right--invariant under $L$. Thus, if we write

\begin{equation}\label{eir-2}
G=\cup_z zL \ \ \text{(disjoint union)},
\end{equation}
then
\begin{equation}\label{eir-3}
 1_{xL}\star 1_{yL}=\sum_z M(x, y, z) \cdot vol \left(L\cap yLy^{-1} \right) \cdot 1_{zL}.
\end{equation}
The sum is of course finite since $Lx^{-1}z\cap yL\neq \emptyset$ implies that $x^{-1}z\in LyL$.  This proves our claim about $\mathcal H\left(G, \mathbb Q\right)$.
We fix such choice of Haar measure and define $\star$ as we explained.

Now, it is obvious that as $\mathbb Q$--vector spaces
\begin{align*}
  &\mathcal H\left(G, L, \mathcal A\right)=\mathcal H\left(G, L, \mathbb Q\right)\otimes_{\mathbb Q}  \mathcal A\\
  & \mathcal H\left(G, \mathcal A\right)=\mathcal H\left(G, \mathbb Q\right)\otimes_{\mathbb Q}  \mathcal A.
  \end{align*}
This enables to define the structure of associative $\mathcal A$--algebra $\mathcal H\left(G, L, \mathcal A\right)$ and $\mathcal H\left(G, \mathcal A\right)$. Furthermore,
\begin{equation}\label{eir-7}
  \epsilon_L=  \epsilon_{L, \calA} = \frac{1_L}{vol(L)}\otimes_{\mathbb Q} 1_{\mathcal A}.
\end{equation}
is the identity of $\mathcal H\left(G, L, \mathcal A\right)$. Furthermore, $\mathcal H\left(G, L, \mathcal A\right)$ is a subalgebra of
$\mathcal H\left(G, \mathcal A\right)$, for all open compact subgroups $L$. 
 We omit  $\otimes 1_{\mathcal A}$ from the notation in this and similar situations in the text that follows.

Let $V$ be a $(\mathcal A, \ G)$--module. Then there exists a unique (subject to the choice of Haar measure above) homomorphism of $\mathcal A$--algebras
$\mathcal H(G, \mathcal A)\longrightarrow \End_{\mathcal A} \left(V\right)$ defined as follows. For $f\in \mathcal H(G, \mathcal A)$, and $v\in V$, we select an open compact
subgroup $L\subset G$
such that $f$ is right invariant by $L$, implying that we can write $f$ as a finite sum $f=\sum_{x} f(x) 1_{xL}$,  and $v\in V^L$. Then, we let
$f.v=vol(L) \cdot \sum_{x}f(x) x.v$. This agrees with the usual
definition $\int_G f(y)y.v dy$ when $\mathcal A=\mathbb C$. Let us show that our definition is correct. Indeed, if   $L'\subset G$ is another open compact subgroup
such that $f$ is right invariant by $L'$, implying that we can write $f$ as a finite sum $f=\sum_{x'} f(x') 1_{x'L'}$,  and $v\in V^{L'}$. We decompose into disjoint unions of right cosets:
$$
L=\cup_{l_1} l_1 L\cap L' \ \  \text{and} \ \ L'=\cup_{l'_1} l'_1 L\cap L'
$$
Then, we have

\begin{align*}
vol(L') \cdot \sum_{x'}f(x') x'.v &= vol(L') \cdot \sum_{x'}\frac{1}{\left[L':L\cap L'\right]}\left( \sum_{l'_1} f(x'l'_1)  \ x'l'_1.v\right)\\
&= vol(L\cap L') \cdot \sum_{x'} \sum_{l'_1} f(x'l'_1)  \ x'l'_1.v\\
&= vol(L\cap L') \cdot \sum_{x} \sum_{l_1} f(xl_1)  \ xl_1.v\\
&= vol(L) \cdot \sum_{x} f(x)  \ x.v.\\
\end{align*}

This shows that the action of elements of $\mathcal H(G, \mathcal A)$ is well--defined. Finally, we check that constructed map
$\mathcal H(G, \mathcal A)\longrightarrow \End_{\mathcal A} \left(V\right)$ is a
homomorphism of $\mathcal A$--algebras.  Indeed, for an arbitrary open compact subgroup $L\subset G$, and $x, y\in G$, we put $f=1_{xL}\otimes 1_{\mathcal A}$ and
$g=1_{yL}\otimes 1_{\mathcal A}$. Then, for
$v\in  V^L$, we remark that
$$
y.v \in L\cap yLy^{-1}.
$$
If we write as a disjoint union

$$
L=\cup_{l_1} \ l_1 \left(L\cap yLy^{-1}\right),
$$
then by definition of the action
$$
fg.v=f.\left(g.v\right)= vol(L) f.\left(y.v\right)=vol(L) vol\left(L\cap yLy^{-1}\right)  \sum_{l_1} xl_1y.v
$$
On the other hand using (\ref{eir-2}) and (\ref{eir-3}), by the definition of the action, we have the following:
$$
f\star g.v=vol(L) \cdot \sum_{\substack{z \ \ \text{as in (\ref{eir-2})}\\ Lx^{-1}z\cap yL\neq \emptyset}}
     vol \left(L\cap yLy^{-1} \right)  z.v= vol(L) vol \left(L\cap yLy^{-1} \right) \sum_{l_1} xl_1y.v
$$

This proves the claim that $\mathcal H(G, \mathcal A)\longrightarrow \End_{\mathcal A} \left(V\right)$ is a
homomorphism of $\mathcal A$--algebras.

As usual (\cite{BZ}, 2.5) $\mathcal H(G, \mathcal A)$--module is non--degenerate if  for any $v\in V$ there exists an open compact subgroup $L\subset G$ such that
(see (\ref{eir-7}))
$$ 
\epsilon_L.v=v.
$$
It easy to check that every  $(\mathcal A, \ G)$--module gives rise to a non--degenerate $\mathcal H(G, \mathcal A)$--module such that 

\begin{equation}\label{eir-4}
x.\left(f.v\right)\overset{def}{=}\left(l_xf\right).v, \ \ f\in \mathcal H(G, \mathcal A), \ \ v\in V,
\end{equation}
where $l_x$ is the left translation $l_xf(y)=f(x^{-1}y)$. Furthermore, it is easy to check the following standard result:

\begin{Lem}\label{eir-5}
  A non--degenerate $\mathcal H(G, \mathcal A)$--module gives rise to a unique $(\mathcal A, \ G)$--module such that (\ref{eir-4}) holds. The category of all
  $(\mathcal A, \ G)$--modules can be identified with the category of all non--degenerate  $\mathcal H(G, \mathcal A)$--modules. 
In particular,  an irreducible $\mathcal H(G, \mathcal A)$--module is also irreducible  $(\mathcal A, \ G)$--module.
\end{Lem}

\vskip .2in
The following lemma is also standard (see \cite{BZ}, Proposition 2.10):

\begin{Lem}\label{eir-6} \begin{itemize}
    \item[(i)] For an irreducible $(\mathcal A, \ G)$--module $V$, and an open compact subgroup $L\subset G$, $\mathcal H\left(G, L, \mathcal A\right)$--module $V^L$ is either $0$ or
      irreducible.
    \item[(ii)] Let $L\subset G$ be an open--compact subgroup. Assume that $V_i$, $i=1,2$, are irreducible   $(\mathcal A, \ G)$--modules such that  $V^L_i\neq 0$, $i=1, 2$. Then, $V_1$ is equivalent to
      $V_2$ as  $(\mathcal A, \ G)$--modules if and only if  $V^L_1$ is equivalent to
      $V^L_2$ as  $\mathcal H\left(G, L, \mathcal A\right)$--modules.    
    \end{itemize}
\end{Lem}
\begin{proof} We just sketch the proof. Let  $L\subset G$ be an open--compact subgroup.
  
Then, $\epsilon_L$ defined in (\ref{eir-7}) is the identity of the associative algebra $\mathcal H\left(G, L, \mathcal A\right)$. Moreover, we have the following:

\begin{equation}\label{eir-8}
\mathcal H\left(G, L, \mathcal A\right)= \epsilon_L\mathcal H\left(G,\mathcal A\right)\epsilon_L
\end{equation}

Now, we sketch the proof of (i). If $0\subsetneqq W\subsetneqq V^L$ is a $\mathcal H\left(G, L, \mathcal A\right)$--submodule of $V_L$. Then, $V_1\overset{def}{=}\mathcal H\left(G, \mathcal A\right)W$ is a
$(\mathcal A, \ G)$--submodule $V$ such that $V_1^L=W$. Since $V$ is irreducible and $V^L\neq W$ this a contradiction. For (ii)  by adjusting the notation we proceed as in the proof of
b) in (\cite{BZ}, Proposition 2.10).
\end{proof}

\vskip .2in
The following theorem is also standard. It is a part of  (\cite{BZ}, Proposition 2.10 c)) but we make it more explicit.

\begin{Thm}\label{eir-9} 
  Let $L\subset G$ be an open--compact subgroup. Then, for each maximal proper left ideal  $I\subset \mathcal H\left(G, L, \mathcal A\right)$,
  there exists a unique  left ideal $J'$ of $\mathcal H\left(G, \mathcal A\right)$ such that the following three conditions hold:
\begin{itemize}
\item[(i)] $J'\subset  \mathcal H\left(G, \mathcal A\right)\epsilon_L$
\item [(ii)]  $I\subset J'$
\item[(iii)] $\mathcal H\left(G, \mathcal A\right)\epsilon_L/J'$ is irreducible.
\end{itemize}
The left ideal $J'$ is a unique maximal proper left--ideal, denoted by  $J_I=J_{I, L}$, in  $\mathcal H\left(G, \mathcal A\right)\epsilon_L$ which contains $I$.
It is a sum of all proper left ideals
in $\mathcal H\left(G, \mathcal A\right)\epsilon_L$ which contain $I$. Moreover, $\epsilon_L\star J_{I, L}= I$. 
\begin{itemize}
\item[(iv)] Regarding
$$
  \calV(I, L)\overset{def}{=} \mathcal H\left(G, \mathcal A\right)\epsilon_L/J_{I, L}
  $$
as a $(\mathcal A, \ G)$--module, we have that its space of $L$--invariants 
   is isomorphic to (irreducible module)  $\mathcal H\left(G,L, \mathcal A\right)/I$ as a
   $\mathcal H\left(G,L, \mathcal A\right)$--module. Up to isomorphism, 
   $\calV(I, L)$ is a unique irreducible $(\mathcal A, \ G)$--module with this property. 
\item[(v)] The $(\mathcal A, \ G)$--module
  $$
  \calW(I, L)\overset{def}{=} \mathcal H\left(G, \mathcal A\right)\epsilon_L/\mathcal H\left(G, \mathcal A\right)I
  $$
  has a unique maximal proper subrepresentation, and the corresponding quotient is
  $\calV(I, L)$. The canonical projection $\calW(I, L)^L\longrightarrow \calV(I, L)^L$ is isomorphism of $\mathcal H\left(G,L, \mathcal A\right)$--modules.
  \item[(vi)] If $f\in \mathcal H\left(G, L, \mathcal A\right)$ does not belong to all maximal left ideals of $\mathcal H\left(G, L, \mathcal A\right)$,
  then there exists  an irreducible $(\mathcal A, \ G)$--module such that $f$ acts as a non--zero operator. More explicitly, if $f\not\in I$,
  then $f$ is not zero on $\calV(I, L)$.
\item[(vii)] The ideal $I\cap \mathcal A\epsilon_L$ is a prime ideal in  $\mathcal A\epsilon_L\equiv \cal A$. The ideal is maximal, if
  $\calA$--module  $\mathcal H\left(G,L, \mathcal A\right)/I$ is finite.
\end{itemize}
\end{Thm}
\begin{proof} If $J$ is a proper left ideal contained in $ \mathcal H\left(G, \mathcal A\right)\epsilon_L$ which contains $I$. Then,
  $\epsilon_L J$ is a left ideal in $\mathcal H\left(G,L, \mathcal A\right)$ which contains $I$. Since $I$ is maximal proper left ideal. Hence,
  $\epsilon_L J=I$ or  $\epsilon_L J=\mathcal H\left(G,L, \mathcal A\right)$. In the latter case, we have
  $$
  J\supset  \mathcal H\left(G, \mathcal A\right) \epsilon_L J= \mathcal H\left(G, \mathcal A\right)  \mathcal H\left(G,L, \mathcal A\right)=\mathcal H\left(G, \mathcal A\right) \epsilon_L.
  $$
  Hence,
  $$
  J=\mathcal H\left(G, \mathcal A\right) \epsilon_L.
  $$
  This is a contradiction. Therefore, if $J_I$  denotes the sum of all proper left ideals $J$ containing $I$, then 
  $$
  \epsilon_L J_I=I.
  $$
  Obviously, $J_I$  satisfies conditions (i) --(iii).  The uniqueness is clear from its construction.
  Of course, we need to establish the existence of at least one such ideal $J$ to be able to define $J_I$. This is easy.
  We just need to take $J=\mathcal H\left(G, \mathcal A\right) I$.

  For (iv),  regarding them   as  $(\mathcal A, \ G)$--modules and  using Lemma \ref{agm-31}, we have
  $$
  \left(\mathcal H\left(G,L, \mathcal A\right)\epsilon_L/J_I\right)^L= \mathcal H\left(G,L, \mathcal A\right)/ \epsilon_L J_I= \mathcal H\left(G,L, \mathcal A\right)/ \epsilon_L I.
  $$
  The uniqueness in the last part of (iv) follows from Lemma \ref{eir-6} (ii).
  Next, (v) is just the reformulation of maximality and uniqueness of $J_I$. (vi) is obvious. We remark that
  maximal left ideals  of $\mathcal H\left(G, L, \mathcal A\right)$ exist by Zorn's lemma. Finally, (vii) follows from Lemma  \ref{pav-5} using simplified arguments of Lemma \ref{agm-400}
  and Theorem \ref{agm-4}.
   \end{proof}

\vskip .2in

\begin{Cor}\label{eir-900} 
  Let $L\subset G$ be an open--compact subgroup. Then, for each irreducible $\mathcal H\left(G, L, \mathcal A\right)$--module $U$ there exists a unique  to an isomorphism
  irreducible $\left(\mathcal A, \ G\right)$--module $V$ such that its space of $L$--invariants is isomorphic to $U$ as $\mathcal H\left(G, L, \mathcal A\right)$--modules.
Furthermore, if the annhilator of a $\calA$--module $U$ is equal to the annhilator of $V$ (see Lemma \ref{agm-400} for the definition of the annhilator).
In addition, if   $U$ is $\calA$--finite, then the annhilator of $V$ is a maximal ideal.  
\end{Cor}
\begin{proof} This first part is immediate from Lemma \ref{eir-6} and Theorem \ref{eir-9}. Next, as in the proof of Lemma \ref{agm-400}, the annhilator $\Ann_{\mathcal A}(U)$
  is a prime ideal, say $\mathfrak p$. Now, the action of  $\mathcal H\left(G, L, \mathcal A\right)$ on $U$ factors through the canonical map
  $\mathcal H\left(G, L, \mathcal A\right)\longrightarrow  \mathcal H\left(G, L, \mathcal A/\mathfrak p \right)$. In this way, we may regards $U$ as a
  $\mathcal H\left(G, L, \mathcal A/\mathfrak p \right)$--module. Now, we use the first part of the proof which guarantees that there exists, unique up to an isomorphism, an irreducible
  $\left(\mathcal A/\mathfrak p, \ G\right)$--module  $V_1$  such that its space of $L$--invariants is isomorphic to $U$ as $\mathcal H\left(G, L, \mathcal A/\mathfrak m\right)$--modules.
  If we regard $V_1$ as a  $\left(\mathcal A, \ G\right)$--module, then we obtain an irreducible module with the space of $L$--invariants isomorphic to $U$ as $\mathcal H\left(G, L,
  \mathcal A\right)$--modules.  Hence $V_1$ is isomorphic to $V$. This clearly implies that the annhilator of $V$ contains  $\mathfrak p$. They are clearly equal or otherwise the annhilator of
  $U$ would be larger.   Finally, the last claim follows from  Theorem \ref{eir-9} (vii).
\end{proof}

\section{An application of Theorem \ref{eir-9}}\label{aeir}

In this section we again assume that  $\mathcal A$ is a $\mathbb Q$--algebra, and $G$ an $l$--group. The goal of this section is to discuss
$$
\End_{(\mathcal A, \ G)}\left(V\right)= \End_{\mathcal H\left(G, \ \calA\right)}\left(V\right),
$$
for an  irreducible $(\mathcal A, \ G)$--module $V$. We also consider
$$
\End_{\mathcal H\left(G, \ L, \ \calA\right)}\left(V^L\right),
$$
for an open compact subgroup $L\subset G$ such that $V^L\neq 0$. It is obvious that the restriction map
gives an embedding
$$
\End_{(\mathcal A, \ G)}\left(V\right)= \End_{\mathcal H\left(G, \ \calA\right)}\left(V\right)\hookrightarrow
\End_{\mathcal H\left(G, \ L, \ \calA\right)}\left(V^L\right).
$$
In general, they are both division algebras central over the field of fractions $k(\mathfrak p)$ of $A/\mathfrak p$ where $\mathfrak p$ is annhilator of $V$ in $\cal A$.
We have the following result (see \cite{cas}, Proposition 2.2.2 for the  proof of the similar result by different means):

\begin{Thm}\label{aeir-1} Assume that $V$ is an  irreducible $(\mathcal A, \ G)$--module. Then, the restriction map  $\End_{(\mathcal A, \ G)}\left(V\right)\longrightarrow
\End_{\mathcal H\left(G, \ L, \ \calA\right)}\left(V^L\right)$ induces an isomorphism of division algebras over $k(\mathfrak p)$. 
  \end{Thm}
\begin{proof} We use Theorem \ref{eir-9}.  We select maximal proper left ideal $I\subset \mathcal H\left(G, L, \mathcal A\right)$ such that we have the following isomorphism of
  $(\mathcal A, \ G)$--modules
  $$
  V\simeq \calV(I, L)
  $$
  Then, 
  $$
  V^L\simeq \mathcal H\left(G,L, \mathcal A\right)/I
  $$
  as $\mathcal H\left(G,L, \mathcal A\right)$--modules.

  Now, we give elementary description of
 $$
\End_{\mathcal H\left(G, \ L, \ \calA\right)}\left(\mathcal H\left(G,L, \mathcal A\right)/I\right).
$$

 First, let $f+I\in \mathcal H\left(G,L, \mathcal A\right)/I$ such that $I\star f\subset I$. Then, the map
 $h+I\longmapsto h\star f+ I$ belongs to $\End_{\mathcal H\left(G, \ L, \ \calA\right)}\left(\mathcal H\left(G,L, \mathcal A\right)/I\right)$. We call this map $\varphi_f$. Conversely, let
$$
\varphi\in
\End_{\mathcal H\left(G, \ L, \ \calA\right)}\left(\mathcal H\left(G,L, \mathcal A\right)/I\right).
$$

If we put $f+I=\varphi(\epsilon_L+I)$, then
  $$
  I\star f+I= I\star (f+I)=I\star \varphi(\epsilon_L+I)=\varphi(I\star \epsilon_L+I)=\varphi(I)=I.
  $$
  Hence, $I\star f\subset I$. Also,
  $$
  \varphi(h+I)= \varphi(h\star \epsilon_L + I) =h\star f+I, \ \ h\in  \mathcal H\left(G,L, \mathcal A\right).
  $$
  Thus, $\varphi=\varphi_f$.   This proves the following lemma:
  \begin{Lem}\label{aeir-2}
    $\mathcal A$--algebra with identity $\epsilon_L+I$  consisting of all $f+I$ such that $I\star f\subset I$ is
    anti--isomorphic to $\End_{\mathcal H\left(G, \ L, \ \calA\right)}\left(\mathcal H\left(G,L, \mathcal A\right)/I\right)$: $f+I\longmapsto \varphi_f$, $\varphi_f\varphi_g=\varphi=\varphi_{g\star f}$.
  \end{Lem}

  Now,we prove the theorem. By the remark before the statement of the theorem it is enough to show that the restriction map is surjective. Let $\varphi\in
\End_{\mathcal H\left(G, \ L, \ \calA\right)}\left(\mathcal H\left(G,L, \mathcal A\right)/I\right)$. By Lemma \ref{eir-2}, we can write $\varphi=\varphi_f$ for some
$f\in \mathcal H\left(G, \ L, \ \calA\right)$ such that $I\star f\subset I$. Using Theorem \ref{eir-9}, we can write
$$
\calV(I, L)\overset{def}{=} \mathcal H\left(G, \mathcal A\right)\epsilon_L/\mathcal H\left(G, \mathcal A\right)J_{I, L},
  $$
where $J_{I, L}$ is a unique maximal proper left ideal in $\mathcal H\left(G, \mathcal A\right)\epsilon_L$. Moreover,
$$
\epsilon_L\star J_{I, L}= I.
$$

After these preparations we define $\psi\in \End_{(\mathcal A, \ G)}\left(\calV(I, L)\right)$ by
$$
\varphi(h+J_{I, L})= h\star f + J_{I, L}, \ \ h\in \mathcal H\left(G, \mathcal A\right)\epsilon_L.
$$
First of all, this map is well--defined since $h-h'\in J_{I, L}$ implies that
$$
(h-h') \star f \in  J_{I, L}\star f.
$$

We observe that $J_{I, L}\star f$ is left ideal in $\mathcal H\left(G, \mathcal A\right)\epsilon_L$. Also, we note that
$$
\epsilon_L\star J_{I, L}\star f =I\star f\subset I.
$$
Consequently, we have the following. The sum $J_{I, L}\star f+ \mathcal H\left(G, \mathcal A\right)I$ is a left ideal in
$\mathcal H\left(G, \mathcal A\right)\epsilon_L$  which contains $I$, and satisfies
$$
\epsilon_L\star\left(J_{I, L}\star f+\mathcal H\left(G, \mathcal A\right)I\right)=I.
$$
This shows that  this ideal is proper ideal in $\mathcal H\left(G, \mathcal A\right)\epsilon_L$, and contains $I$. Thus, it is contained in $J_L$.
In particular, we have
$J_{I, L}\star f\subset J_{I, L}$. Hence, $(h-h')\star f\in J_{I, L}$. This shows that $\varphi$ is well--defined. Obviously, it belongs to
$\End_{(\mathcal A, \ G)}\left(\calV(I, L)\right)$. Finally, the space of $L$--invariants in $\calV(I, L)$ is equal to

$$
\epsilon_L\star  \calV(I, L)=\epsilon_L\mathcal H\left(G, \mathcal A\right)\epsilon_L/J_{I, L}\simeq \mathcal H\left(G,L, \mathcal A\right)/I.
$$
The isomorphism is $h+J_{I, L}\longmapsto h+I$,  for $h\in H\left(G, L, \mathcal A\right)$, and it is an isomorphism of $\mathcal H\left(G,L, \mathcal A\right)$--modules.
We transfer $\varphi$ via that isomorphism to $\epsilon_L\star \calV(I, L)$. As a result,  we obtain the following map:
$$
h+J_{I, L}\longmapsto h\star f+J_{I, L},
$$
which is clearly the restriction of $\psi$. \end{proof}

\section{Another application of Theorem \ref{eir-9}}\label{aaeir}

The aim of this section is to prove the following theorem:

\begin{Thm}\label{aaeir-1} Assume that $\mathcal A$ is a field  (and then extension of $\mathbb Q$). Let $G$ be an $l$--group and $L\subset G$ an open compact subgroup.
  Let $V$ be an irreducible $(\mathcal A, \ G)$--module such that $V^L\neq 0$ and  $\mathcal A$--finite dimensional (i.e., $V^L$ is an $\mathcal A$--admissible 
  irreducible $\mathcal H\left(G, \ L, \ \mathcal A\right)$--module). Then,  for any field extension $\mathcal A\subset \mathcal B$, there exists   irreducible  $(\mathcal B, \ G)$--modules
  $V_1, \ldots, V_t$ such that the following holds:
  \begin{itemize}
  \item[(i)] $V^L_i\neq 0$ for all  $1\le i\le t$. 
  \item[(ii)] $V^L_i$ are  $\mathcal B$--admissible  irreducible $\mathcal H\left(G, \ L, \ \mathcal B\right)$--modules.
  \item[(iii)] $V_{\mathcal B}\overset{def}{=}\mathcal B \otimes_{\mathcal A} V \simeq V_1\oplus \cdots \oplus V_t$ as $(\mathcal B, \ G)$--modules.
    \end{itemize}
\end{Thm}
\begin{proof} First, we recall that $\mathcal H\left(G, \ L, \ \mathcal A\right)$ is an
 associative $\mathcal A$--algebra with identity $\epsilon_{L, \mathcal A}$
(see \ref{eir-7}). We have 
$$
\mathcal H\left(G, \ L, \ \mathcal B\right)=  \mathcal B \otimes_{\mathcal A} \mathcal H\left(G, \ L, \ \mathcal A\right),
$$
and
$$
\epsilon_{L, \mathcal B}= 1 \otimes_{\mathcal B} \epsilon_{L, \mathcal A}.
$$
Next, by Lemma \ref{agm-7} (i), we have
  $$
  \left(\mathcal B\otimes_{\mathcal A} V\right)^L= \mathcal B\otimes_{\mathcal A} V^L.
  $$
  Next, since $V$ is irreducible and $V^L\neq 0$, we conclude that  $V^L$ is an irreducible $\mathcal H\left(G, \ L, \ \mathcal A\right)$--module (see Lemma \ref{eir-6} (i)).
  Put
  $$
  W=V^L,
  $$
  and
  $$
  W_{\mathcal B}= \mathcal B\otimes_{\mathcal A} W.
  $$
  Obviously, the later is a $\mathcal B$--admissible module for $\mathcal H\left(G, \ L, \ \mathcal B\right)$.

 Let $$
  \varphi_{\mathcal A, W}: \mathcal H\left(G, \ L, \ \mathcal A\right)\longrightarrow
\End_{\mathcal A}(W)
$$
be the corresponding homomorphism of $\mathcal A$--algebras.  We let $\mathcal H_{\mathcal A, W}$ be the image of $\varphi_{\mathcal A, W}$. Similar notation we introduce for the field $\mathcal B$.
Then,  we may take
$$
\varphi_{\mathcal B, W_{\mathcal B}}= id_{\mathcal B} \otimes_{\mathcal B} \varphi_{\mathcal A, W}.
$$

Next, by Schur's lemma, we have that
  \begin{equation}\label{aaeir-200}
  \mathcal D\overset{def}{=} \End_{\mathcal H\left(G, \ L, \ \calA\right)}\left(W\right)
  \end{equation}
which center contained $\mathcal A$. Since, by the assumption $V^L$ is $\mathcal A$--finite dimensional, we conclude that $D$ is finite $\mathcal A$--dimensional.
Hence, we have the following standard result:

\begin{Lem}\label{aaeir-2}
  Maintaining above assumptions, we have the following:
  \begin{itemize}
  \item[(i)] $\mathcal H_{\mathcal A, W}$ is simple $\mathcal A$--algebra; its unique simple module up to an isomorphism is
    $W$.
  \item[(ii)]  $\mathcal H_{\mathcal A, W}= \End_{\mathcal D}(W)$.
   \item[(iii)]   $\mathcal H_{\mathcal B, W_{\mathcal B}}= \mathcal B\otimes_{\mathcal A} \End_{\mathcal D}(W)$ is a semisimple $\mathcal B$--algebra.
  \end{itemize}
  \end{Lem}
\begin{proof} (ii) is a consequence of Jacobson's density theorem (known as a Wedderburn's theorem, see \cite{lang}, Chapter XVII, Corollary 3.5).  (i) is well--known once we have (ii)
  (see \cite{lang}, Chapter XVII, Theorem 5.5). For  (iii), we note that  (\cite{lang}, Chapter XVII, Theorem 6.2) implies that
  $\mathcal B\otimes_{\mathcal A} \End_{\mathcal D}(W)$ is a semisimple $\mathcal B$--algebra. Finally, we have

\begin{align*}
  \mathcal H_{\mathcal B, W_{\mathcal B}} &=\varphi_{\mathcal B}\left(\mathcal H\left(G, \ L, \ \mathcal A\right)\right)\\
  & = id_{\mathcal B} \otimes_{\mathcal B} \varphi_{\mathcal A, W} \left(\mathcal B\otimes_{\mathcal A} \mathcal H\left(G, \ L, \ \mathcal A\right)\right)\\
  &=   \mathcal B\otimes_{\mathcal A}  \mathcal H_{\mathcal A, W}\\
  &= \mathcal B\otimes_{\mathcal A} \End_{\mathcal D}(W). 
\end{align*}
This completes the proof of (iii). 
\end{proof}

As a corollary of Lemma \ref{aaeir-2} (iii),  there exists $\mathcal B$--admissible modules $W_1, W_2, \ldots W_t$ of $\mathcal H_{\mathcal B, W_{\mathcal B}}$ (and consequently of
$\mathcal H\left(G, \ L, \ \mathcal B\right)$) such that

\begin{equation}\label{aaeir-3}
\mathcal B\otimes_{\mathcal A} V^L= \mathcal B\otimes_{\mathcal A} W=  W_{\mathcal B}\simeq W_1\oplus W_2 \oplus \cdots \oplus W_t
\end{equation}
as $\mathcal H\left(G, \ L, \ \mathcal B\right)$--modules.

Now, we apply Theorem \ref{eir-9}. Select $v\in V^L$, $v\neq 0$, and decompose it according to the decomposition in (\ref{aaeir-3}):
\begin{equation}\label{aaeir-30}
v=\sum_{i=1}^t \ w_i \ \ w_i\in W_i.
\end{equation}
We let

\begin{align*}
&  I\overset{def}{=} \Ann_{\mathcal H\left(G, \ L, \ \mathcal A\right)}(v),  \ \ V\simeq \mathcal H\left(G, \ L, \ \mathcal A\right)/I\\
&   I_i\overset{def}{=} \Ann_{\mathcal H\left(G, \ L, \ \mathcal B\right)}(w_i),  \ \ W_i\simeq \mathcal H\left(G, \ L, \ \mathcal B\right)/I_i, \ \ 1\le i\le t.\\
\end{align*}

\begin{Rem}\label{aaeir-4}
  In what follows we use repeatedly the following elementary observation. Let $X$ and $Y$ be non--zero vector spaces over the field $\mathcal A$.
  Let $Z\subset X$, $Z\neq 0$,  be a subspace. Then, if $\sum_{i=1}^l x_i\otimes y_i\in Z\otimes_{\mathcal A} Y$, with  $\mathcal A$--linearly independent vectors $y_1, \ldots, y_l$,
  then $x_1,\ldots, x_l\in Z$. Indeed, if $\alpha$ is $\mathcal A$--linear functional, then there exists an $\mathcal A$--linear map $X\otimes_{\mathcal A} Y\longrightarrow X$
  such that $x\otimes y\longmapsto \alpha(y)x$. It maps  $Z\otimes_{\mathcal A} Y$ into $Z$.
  Now, since  $y_1, \ldots, y_l$ are $\mathcal A$--linearly independent, there exists linear functionals $\alpha_1, \ldots, \alpha_t$ on $Y$ such that
  $\alpha_i(y_j)=\delta_{ij}$ (a Kronecker delta). Consequently, $\alpha_k\left(\sum_{i=1}^l x_i\otimes y_i\right)=x_k\in Z$.
\end{Rem}

\vskip .2in 
\begin{Lem}\label{aaeir-5}
  $\Ann_{\mathcal H\left(G, \ L, \ \mathcal A\right)}(1\otimes v)= \mathcal B\otimes_{\mathcal A} I=I_1\cap I_2\cap \cdots \cap I_t$.
\end{Lem}
\begin{proof} $\Ann_{\mathcal H\left(G, \ L, \ \mathcal A\right)}(1\otimes v)=I_1\cap I_2\cap \cdots \cap I_t$ is obvious from (\ref{aaeir-3}) and (\ref{aaeir-30}). Also,
  $ \mathcal B\otimes_{\mathcal A} I\subset \Ann_{\mathcal H\left(G, \ L, \ \mathcal A\right)}(1\otimes v)$ is obvious. The converse inclusion follows from elementary
  Remark \ref{aaeir-4}.
  \end{proof}

Now, following Theorem \ref{eir-9}, we construct maximal left ideals  

\begin{align*}
  & J\overset{def}{=}\sum_{\substack{J'\subset \mathcal H\left(G, \ \mathcal A\right)  \epsilon_{L, \mathcal A} \ \text{a left ideal} \\
      \epsilon_{L, \mathcal A} J'=I}} \ J'  \subset \mathcal H\left(G,  \ \mathcal A\right) \epsilon_{L, \mathcal A}  \\
  & \\
  &  J_i\overset{def}{=}\sum_{\substack{J'\subset \mathcal H\left(G,\ \mathcal B\right)  \epsilon_{L, \mathcal B} \ \text{a left ideal} \\
      \epsilon_{L, \mathcal B} J'=I_i}} \ J' \subset \mathcal H\left(G,  \ \mathcal B\right) \epsilon_{L, \mathcal B}, \ \ 1\le i\le t.   \\
  \end{align*}

Then, we have (see Theorem \ref{eir-9} (iv))
$$
V\simeq  \calV_{\mathcal A}(I, L)\overset{def}{=}\mathcal H\left(G, \  \mathcal A\right)/J.
$$
Consequently, since $\mathcal B$ is a field, we have

\begin{equation}\label{aaeir-6}
  \mathcal B\otimes_{\mathcal A} V\simeq \mathcal H\left(G, \ \mathcal B\right)/\mathcal B\otimes_{\mathcal A} J 
\end{equation}

We also define irreducible $(\mathcal B, \ G)$--modules using  (Theorem \ref{eir-9} (iv))
$$
V_i\overset{def}{=}\mathcal H\left(G, \  \mathcal B\right)/J_i, \ \ 1\le i\le t. 
$$
By   Theorem \ref{eir-9} (iv), we have
$$
V_i^L=  \mathcal H\left(G, \  \mathcal B\right)/I_i\simeq W_i 
$$
as $\mathcal H\left(G, \ L, \  \mathcal B\right)$--modules for all $1\le i\le t$. Thus, $V_1, V_2, \ldots, V_t$ satisfies (i) and (ii) of the theorem. It remains to prove (iii). 
We  need the following lemma:

\vskip .2in 
\begin{Lem}\label{aaeir-7}
   $\mathcal B\otimes_{\mathcal A} J=J_1\cap J_2\cap \cdots \cap J_t$.
\end{Lem}
\begin{proof} We prove $\mathcal B\otimes_{\mathcal A} J\subset J_i$ for all $i=1, \ldots, t$. Indeed, let $J'\subset \mathcal H\left(G, \ L, \ \mathcal A\right)  \epsilon_{L, \mathcal A} $
  be a left ideal such that $\epsilon_{L, \mathcal A} J'=I$. Then, we define a left ideal in $ \mathcal H\left(G, \ L, \ \mathcal B\right)  \epsilon_{L, \mathcal B} $ as follows:
  $$
  J{''}_i\overset{def}{=}\mathcal H\left(G, \  \mathcal B\right) \star I_i+  \mathcal B\otimes_{\mathcal A} J'.
  $$
  Then, applying Lemma \ref{aaeir-5}, we obtain 
  $$
  \epsilon_{L, \mathcal B}J{''}_i=  \epsilon_{L, \mathcal B} \left(\mathcal H\left(G, \  \mathcal B\right) \star I_i+  \mathcal B\otimes_{\mathcal A} J'\right)=
  I_i+  \mathcal B\otimes_{\mathcal A}  \epsilon_{L, \mathcal A}J' = I_i +   \mathcal B\otimes_{\mathcal A}  I= I_i, 
  $$
  for all $1\le i\le t$.  Consequently, we have
  $$
  \mathcal B\otimes_{\mathcal A} J'\subset J{''}\subset J_i, \ \ 1\le i\le t.
  $$
  Since $J'$ is arbitrary, we obtain
  $$
  \mathcal B\otimes_{\mathcal A} J\subset J_i, \ \ 1\le i\le t.
  $$
  This proves
  $$
  \mathcal B\otimes_{\mathcal A} J\subset J_1\cap J_2\cap \cdots \cap J_t.
  $$

  Conversely, let $f\in J_1\cap J_2\cap \cdots \cap J_t$. Then, we define a left ideal
  $$
  J{''}\overset{def}{=}  \mathcal H\left(G,  \ \mathcal B\right) f\subset J_1\cap J_2\cap \cdots \cap J_t.
  $$
  Then, for each $i$, we have
  $$
  \epsilon_{L, \mathcal B} J{''}\subset I_i,
  $$
  by the definition of ideals $J_i$ and an argument as above with $J_i$. Hence, by   Lemma \ref{aaeir-5}, we obtain
  
  \begin{equation}\label{aaeir-8}
  \epsilon_{L, \mathcal B} J{''}\subset \mathcal B\otimes_{\mathcal A} I.
  \end{equation}

  Now, we write
  $$
  f=\sum_{i=1}^l b_i\otimes f_i, \ \  f_i\in  \mathcal H\left(G,  \ \mathcal A\right), b_i \in \mathcal B,
  $$
  with $b_1, \ldots, b_l$ are $\mathcal A$---linearly independent. 
  Then, (\ref{aaeir-8}) implies that
 $$
  \sum_{i=1}^l b_i\otimes \epsilon_{L, \mathcal B}  F\star f_i \in  \mathcal B\otimes_{\mathcal A} I,
  $$
  for any $F\in \mathcal  \mathcal H\left(G,  \ \mathcal A\right)$. Applying now Remark \ref{aaeir-4} we obtain
  $$
  \epsilon_{L, \mathcal B}  F\star f_i \in I,
  $$
  for all  $F\in \mathcal  \mathcal H\left(G,  \ \mathcal A\right)$ and all $i$. This implies that
  $$
  f_i\in  \mathcal H\left(G,  \ \mathcal A\right)f_i \subset J,
  $$
  for all $i$. Consequently, we obtain
  that
  $$
  f= \sum_{i=1}^l b_i\otimes f_i\in \mathcal B\otimes_{\mathcal A} J.
  $$
  This proves that
  $$
  J_1\cap J_2\cap \cdots \cap J_t\subset  \mathcal B\otimes_{\mathcal A} J.
  $$
  The proof of lemma is complete.
\end{proof}

\vskip .2in

Now, we are ready to prove (iii) in the theorem, and thus complete the proof of the theorem. By (\ref{aaeir-6}) and Lemma \ref{aaeir-7}, we have the following inclusion of
$(\mathcal B, \ G)$--modules:

$$
\mathcal B\otimes_{\mathcal A} V\hookrightarrow V_1\oplus V_2 \oplus \cdots \oplus V_t.
$$
But the map is surjective since the map is surjective on level of $L$--invariants by counting $\mathcal A$--dimensions (see (\ref{aaeir-3}))  which implies the following:

$$
\mathcal B\otimes_{\mathcal A} V = \mathcal H\left(G, \ \mathcal B\right) \left(\mathcal B\otimes_{\mathcal A} V^L\right)=
\sum_{i=1}^t \mathcal H\left(G, \ \mathcal B\right)W_i= \oplus_{i=1}^t V_i.
$$
This completes the proof of the theorem. 

\end{proof}

\section{Applications and Improvements of Theorem \ref{aaeir-1}}\label{ebf}

We start this section with the following application of Theorem \ref{aaeir-1}:

\begin{Cor}\label{ebf-0} Assume that $\mathcal A$ is any subfield of $\mathbb C$.
  Let $G$ be a reductive $p$--adic group (i.e.,  a group of $k$--points of
  a reductive group over a local non--Archimedean field $k$). Let $L\subset G$ be an open compact subgroup.
  Let $V$ be an irreducible $(\mathcal A, \ G)$--module such that $V^L\neq 0$ and  $\mathcal A$--finite dimensional. Then, $V$ is $\mathcal A$--admissible (see Definition \ref{agm-2}).
\end{Cor}
\begin{proof} We may assume that $\mathcal A\subset \mathbb C$. Then, in Theorem \ref{aaeir-1} we select $\mathcal B=\mathbb C$.
  Then all $W_i$, $1\le i\le t$, are irreducible smooth complex representations
  of a reductive $p$--adic group $G$. Then, by a result of Jacquet (\cite{renard}, Theorem VI.2.2), every representation $W_i$ is $\mathcal C$--admissible. This implies that
  $\mathbb C\otimes_{\mathcal A} V$ is. Hence, for every open compact subgroup $L_0\subset G$, the complex vector space
  $\left(\mathbb C\otimes_{\mathcal A} V\right)^{L_0}$ is finite dimensional. But, by Lemma \ref{agm-7} (i), we have
  $$
  \mathbb C\otimes_{\mathcal A} V^{L_0} \simeq \left(\mathbb C\otimes_{\mathcal A} V\right)^{L_0}.
  $$
  But then
  $$
  \dim_{\mathcal A}  V^{L_0} = \dim_{\mathbb C}  \left(\mathbb C\otimes_{\mathcal A} V\right)^{L_0},
  $$
  proving the corollary.
  \end{proof}

\vskip .2in

\vskip .2in
The following is analogue of the result for finite dimensional representations of associative algebras  (see \cite{currei}, Section 29).

\begin{Cor}\label{ebf-00}  Assume that $\mathcal A$ is a field  (and then extension of $\mathbb Q$). Let $G$ be an $l$--group and $L\subset G$ an open compact subgroup.
  Assume that  $V$ and $U$ are irreducible $(\mathcal A, \ G)$--modules such that $V^L\neq 0$, $U^L\neq 0$,  and  both
  $\mathcal A$--admissible,  for any field extension $\mathcal A\subset \mathcal B$, if $V_{\mathcal B} $ and $U_{\mathcal B}$ have disjoint Jordan--H\" older series, then
  $V\simeq U$ as $(\mathcal A, \ G)$--modules.
\end{Cor}
\begin{proof} By Theorem \ref{aaeir-1} (iii), both representations  $V_{\mathcal B} $ and $U_{\mathcal B}$ are semisimple and have finite length. By  Theorem \ref{aaeir-1} (i) and (ii),
  every irreducible composition factor has non--zero and $\mathcal A$--finite dimensional space of $L$--invariants. Consequently, both
  $V^L_{\mathcal B} $ and $U^L_{\mathcal B}$ are semi--simple. Thus, if they have common irreducible factor,
  then
  $$
  \Hom_{\mathcal H\left(G, \ L, \ \mathcal B\right)}(V^L_{\mathcal B}, \  U^L_{\mathcal B})\neq 0.
  $$
  But by the  results that can be found in (\cite{currei}, Section 29):
  $$
  \Hom_{\mathcal H\left(G, \ L, \ \mathcal B\right)}(V^L_{\mathcal B}, \  U^L_{\mathcal B})  \simeq \mathcal B\otimes_{\mathcal A}\Hom_{\mathcal H\left(G, \ L, \ \mathcal A\right)}(V^L, \ U^L).
  $$
  Thus, we obtain
  $$
  \Hom_{\mathcal H\left(G, \ L, \ \mathcal A\right)}(V^L, \ U^L)\neq 0.
  $$
Then, Lemma \ref{eir-6} (ii) implies that $V\simeq U$.  
  \end{proof}

\vskip .2in
Another application of Theorem \ref{aaeir-1}  is the following corollary: 

\begin{Cor}\label{eib-99} Assume that $\mathcal A$ is a field  (and then extension of $\mathbb Q$). Let $G$ be an $l$--group and $L\subset G$ an open compact subgroup.
  Let $V$ be an irreducible $(\mathcal A, \ G)$--module such that $V^L\neq 0$ and  $\mathcal A$--admissible.  Then, $V$ is absolutely irreducible (i.e.,
  $V_{\mathcal B}$ is irreducible for all field extensions $\mathcal A\subset \mathcal B$, see  \cite{cas} and \cite{robert}) if and only if
  $\End_{(\mathcal A, \ G)}\left(V\right)=\mathcal A$.
\end{Cor}
\begin{proof} Assume that $\End_{(\mathcal A, \ G)}\left(V\right)=\mathcal A$. Then, using the notation of Lemma \ref{aaeir-2},  $\mathcal H_{\mathcal A, W}= \End_{\mathcal A}(W)$, where $W=V^L$. Thus, if
$\mathcal A\subset \mathcal B$ is a filed extension,using Lemma \ref{aaeir-2} (ii),  we obtain   
$$
\mathcal H_{\mathcal B, W_{\mathcal B}}= \mathcal B\otimes_{\mathcal A} \End_{\mathcal A}(W)= \End_{\mathcal B}(W_{\mathcal B}).
$$
This implies that $W_{\mathcal B}$ is irreducible $\mathcal H\left(G, \ L, \ \mathcal B\right)$--module. Applying Theorem \ref{aaeir-1} we conclude that $V_{\mathcal B}$ is irreducible.

Conversely, assume that $V$ is absolutely irreducible.  Then obviously $W=V^L$  is absolutely irreducible $\mathcal A$--admissible $\mathcal H\left(G, \ L, \ \mathcal A\right)$--module
(see Lemma \ref{eir-6}). Now, we apply the following lemma (\cite{currei}, Section 29):

 \begin{Lem}\label{ebf-4}
   Assume that $U$ is an irreducible $\mathcal A$--admissible $\mathcal H\left(G, \ L, \ \mathcal A\right)$--module. Then, $U$ is absolutely irreducible if and only if
   $\End_{\mathcal H\left(G, \ L, \ \mathcal A\right)}(U)=\mathcal A$.
    \end{Lem}
Finally, Theorem \ref{aeir-1} completes the proof. 
\end{proof}

\vskip .2in
We remark that ff $V$ is absolutely irreducible, then  $V_{\mathcal B}$ is also absolutely irreducible module. One needs to apply Corollary \ref{eib-99} and the observation:
$$
V_{\mathcal C}=\mathcal C \otimes_{\mathcal A} V\simeq \mathcal C \otimes_{\mathcal B} \left(\mathcal B \otimes_{\mathcal A} V \right)\simeq \mathcal C \otimes_{\mathcal B} V_{\mathcal B},
$$
for field extensions $\mathcal A\subset \mathcal B\subset \mathcal C$.

\vskip .2in
Finally, we give an improvement of Theorem \ref{aaeir-1}. 

\begin{Cor}\label{eib-100} Assume that $\mathcal A$ is a field  (and then extension of $\mathbb Q$). Let $G$ be an $l$--group and $L\subset G$ an open compact subgroup.
  Let $V$ be an irreducible $(\mathcal A, \ G)$--module such that $V^L\neq 0$ and  $\mathcal A$--admissible. Then, we can extend $V$ to an obvious $(\mathcal K, \ G)$--module, say $V^{ext}$, where
  $\mathcal K$ is the center of the division algebra $\End_{(\mathcal A, \ G)} (V)$. Then,
for any field extension $\mathcal K\subset \mathcal B$, there exists  a unique irreducible $(\mathcal B, \ G)$--module $V(\mathcal B)$ such that 
the following holds:
  \begin{itemize}
  \item[(i)] $V^{ext}(\mathcal B)^L\neq 0$. 
  \item[(ii)] $V^{ext}(\mathcal B)^L$ is $\mathcal B$--admissible  irreducible $\mathcal H\left(G, \ L, \ \mathcal B\right)$--module.
  \item[(iii)] $V^{ext}_{\mathcal B}\overset{def}{=} \mathcal B \otimes_{\mathcal A} V^{ext}$ is direct sum of finite number of copies of $V^{ext}(\mathcal B)$.
  \end{itemize}
  In addition, if $\mathcal F$ is a maximal subfield of $\mathcal D$ (which must contain $\mathcal K$)  then $V^{ext}(\mathcal F)$ is absolutely irreducible. 
\end{Cor}
\begin{proof}  This follows from Theorem \ref{aaeir-1} and Corollary \ref{eib-99} but we need some preparation.
  We warn the reader that we use notation from the first part of the proof of Theorem  \ref{aaeir-1}
  freely (see Lemma \ref{aaeir-2}).  As in the proof of Theorem \ref{aaeir-1}, we write $W=V^L$.  Applying Theorem \ref{aeir-1}, we obtain (see (\ref{aaeir-200}))
the following isomorphism of $\mathcal A$--algebras (see (\ref{aaeir-200})):
  $$
  \End_{(\mathcal A, \ G)} (V)\simeq \mathcal D.
  $$
In particular, $\mathcal K$ is the center of $D$. We let
$$
W^{ext} = \left(V^{ext}\right)^L.
$$
Moreover, we have the following isomorphism: 
  $$
  \End_{\mathcal H\left(G, \ L, \ \mathcal K\right)}\left(W^{ext}\right)=\End_{\mathcal H\left(G, \ L, \ \mathcal A\right)}(W)=\mathcal D.
  $$

In difference to what we have in the proof of Theorem \ref{aaeir-1} (see the statement of Lemma \ref{aaeir-2}), the simple algebra
$$
\mathcal H_{\mathcal K, W}= \End_{\mathcal D}(W^{ext})
$$
has center of $\mathcal K$. Thus, by (see \cite{currei}, Section 68), 
$$
\mathcal H_{\mathcal B, W^{ext}_{\mathcal B}}= \mathcal B\otimes_{\mathcal K} \End_{\mathcal D}(W^{ext})
$$
is simple $\mathcal B$--algebra. This observation is responsable for the existence of unique $V(\mathcal B)$.
We leave the details to the reader.

It remains to prove that $V^{ext}(\mathcal F)$ is absolutely irreducible.  We need the following lemma:

\begin{Lem}\label{ebf-2}
 $\mathcal F$--algebra $\mathcal F \otimes_{ \mathcal K}   \mathcal D$ is isomorphic to the  $\mathcal F$--algebra of all
  matrices of size $t\times t$ with coefficients in $\mathcal F$ where $t=\dim_{\mathcal K} \mathcal D$. 
    \end{Lem}
\begin{proof} This is a part of standard theory of simple algebras (see \cite{currei}, Section 68). 
 \end{proof}

 As in the proof of Corollary \ref{ebf-00},  by the results of \cite{currei}, Section 29), we have
$$
\End_{\mathcal H\left(G, \ L, \ \mathcal F\right)}\left(\left(V^{ext}_{\mathcal F}\right)^L\right)\simeq \mathcal F\otimes_{\mathcal K}
\End_{\mathcal H\left(G, \ L, \ \mathcal K\right)}\left(\left(V^{ext}\right)^L\right)=\mathcal F \otimes_{ \mathcal K}   \mathcal D.
$$
Thus, by Lemma \ref{ebf-2}, we see that
$$
\End_{\mathcal H\left(G, \ L, \ \mathcal F\right)}\left(\left(V^{ext}_{\mathcal F}\right)^L\right)
$$
is a matrix algebra of size  $t\times t$ with coefficients in $\mathcal F$. Since, by already proved part (iii) of the corollary,
the module $\left(V^{ext}_{\mathcal F}\right)^L$ is a direct sum of finite number of copies of $\left(V^{ext}(\mathcal F)\right)^L$, we conclude that
the number of copies is equal to $t$ and 
$$
\End_{\mathcal H\left(G, \ L, \ \mathcal F\right)}\left(\left(V^{ext}(\mathcal F)\right)^L\right)=\mathcal F.
$$
Finally, Theorem \ref{aeir-1} and Corollary \ref{eib-99} complete the proof. 
\end{proof}

\section{An Example: Construction of Unramified Irreducible Representions}\label{cuir}

Let $k$ be a non--Archimedean local field. Let $\mathcal O\subset k$ be its ring of integers, and
let $\varpi$ be  a generator of the maximal ideal in $\calO$. Let $q$ be the number of elements in the residue field $\mathcal O/ \varpi\mathcal O$.
Let $G$  be a  $k$--split Zariski connected reductive group. To simplify notation we write $G$ for the group $G(k)$ of $k$--points. Similarly, we do for other algebraic subgroups defined over $k$.

Let 
$$K\overset{def}{=}G(\mathcal O)
$$
is a hyperspecial
maximal compact subgroup of $G$ (\cite{Tits}, 3.9.1). We normalize a Haar measure on $G$  such that $\int_{K}dg=1$  (see Section \ref{eir}).

We recall the  structure of the algebra
$$
\mathcal H\left(G, \  K, \ \mathbb C\right)
$$
is  obtained via Satake isomorphism \cite{cartier}. In more detail,  let $A$ be  a maximal $k$--split torus of $G$.
Let $X^*(A)$ (resp., $X_*(A)$) be the group of $k$--rational
characters (resp., cocharacters) of $A$. Let $W$ be the the Weyl group of
$A$ in $G$. The group $W$ acts on $A$ and its complex dual
torus $\hat{A}$.  The Satake
isomorphism  enables us to identify
$\mathcal H\left(G, \  K, \ \mathbb C\right)$  with the algebra of $W$--invariants
$$
\mathbb C[X^*(\hat{A})]^{W}
$$
where
$$
\mathbb C[X^*(\hat{A})]
$$
is the $\mathbb C$-group algebra of finitely generated free Abelian group.  This is also  the algebra of regular functions
on complex algebraic torus $\hat{A}$. The action of $W$ on the torus is algebraic, and therefore
Let
$$
X\overset{def}{=}\hat{A}/W,
$$
is the complex affine variety  of $W$--orbits in $\hat{A}$. Its algebra of regular functions is
$$
\mathbb C[X]=  \mathbb C[X^*(\hat{A})]^{W}.
$$
Thus, the Satake isomorphism identifies $\mathcal H\left(G, \  K, \ \mathbb C\right)$ with $\mathbb C[X]$  (it depends
on the choice of a  Borel subgroup $B=AU$ of $G$, where $U$ is the unipotent radical).

By standard Nullstellensatz, a point $x\in X$ defines a maximal ideal in $\mathfrak m_x$ in $\mathcal H\left(G, \  K, \ \mathbb C\right)$. 
Then, we apply Theorem \ref{eir-9} to construct irreducible (admissible) $(\mathbb C, \ G)$--module unramified representations  $\calV(\mathfrak m_x, K)$. We have
\begin{align*}
  \calV(\mathfrak m_x, K)^K &\simeq \mathcal H\left(G, \  K, \ \mathbb C\right)/ \mathfrak m_x\\
                           &\simeq \mathbb C \ \ \text{a one dimensional module form the evaluation of $\mathbb C[X]$ at $x$.}                   
\end{align*}
as 
$\mathcal H\left(G, \  K, \ \mathbb C\right)$--modules. Different $x\in X$ give rise to non--isomorphic $\calV(\mathfrak m_x, K)$ $(\mathbb C, \ G)$--modules. 
This completes the description of  complex unramified representations in terms of Hecke algebra $ \mathcal H\left(G, \  K, \ \mathbb C\right)$.

\vskip .2in
By a careful analysis of $\mathbb Z$--structure of Satake isomorphim \cite{gross} due to Gross,
we obtain the following:

\begin{Lem}\label{cuir-1} Let $\mathcal A$ be field which is any extension of $\mathbb Q$ if $G$ is simply--connected, or just an extension of $\mathbb Q(q^{1/2})$ otherwise. Then, we have
  the following isomorphism of $\mathcal A$--algebras. 
  $$
  \mathcal H\left(G, \  K, \ \mathcal A\right)\simeq  \mathcal A\otimes_{\mathbb Q} \ \mathbb Q[X^*(\hat{A})]^{W}= \mathcal A[X^*(\hat{A})]^{W}.
  $$
\end{Lem}

\vskip .1in
Since $\hat{A}$ is a split torus, it is defined over $\mathbb Q$ (and consequently all extensions of $\mathbb Q$) by considering the group algebra $\mathbb Q[X^*(\hat{A})]$. The action of $W$
on $\hat{A}$ preserves $\mathbb Q[X^*(\hat{A})]$ and consequently it is defined over $\mathbb Q$. This implies  that the variety $X$ is defined over $\mathbb Q$ via
$\mathbb Q[X^*(\hat{A})]^{W}$.

\vskip .2in
Now, we prove the main result of this section and of the paper.

\begin{Thm}\label{cuir-4} Let $k$ be a non--Archimedean local field. Let $\mathcal O\subset k$ be its ring of integers, and
  let $\varpi$ be  a generator of the maximal ideal in $\calO$. Let $q$ be the number of elements in the residue field $\mathcal O/ \varpi\mathcal O$.
  Assume that is $G$  is a  $k$--split Zariski connected reductive group. Let $A$ be its maximal $k$--split torus, and $W$ the corresponding group.
  We write $\hat{A}$ for the complex torus dual to $A$. Let $W$ be the Weyl group of $A$ in $G$. The orbit space
  $$
 X\overset{def}{=}\hat{A}/W
 $$
is an affine variety defined over $\mathbb Q$.  Let $K=G(\mathcal O)$ be its  hyperspecial
maximal compact subgroup of $G$.  We normalize a Haar measure on $G$  such that $\int_{K}dg=1$  (see Section \ref{eir}). Let $\overline{\mathbb Q}$ be the algebraic closure of $\mathbb Q$ inside $\mathbb C$.
Let $\mathcal A$ be any  subfield of $\overline{\mathbb Q}$ if $G$ is simply--connected, or a extension  of $\mathbb Q(q^{1/2})$ in $\overline{\mathbb Q}$ otherwise.  We define the (commutative) Hecke algebra
$\mathcal H\left(G, \  K, \ \mathcal A\right)$ with respect to above fixed Haar measures. Then, we have the following:

\begin{itemize}
\item[(i)] (Satake isomorphims over subfields of $\overline{\mathbb Q}$)
  Maximal ideals in $\mathcal H\left(G, \  K, \ \mathcal A\right)$ are parame\--trized by points in $X(\overline{\mathbb Q})$ such that  points in $X(\overline{\mathbb Q})$
  give the same maximal ideal  if and only if they are  $Gal(\overline{\mathbb Q}/\mathcal A)$--conjugate: for $x\in  X(\overline{\mathbb Q})$, we denote by
  $\mathfrak m_{x, \mathcal A}$ the corresponding maximal ideal.   The corresponding
  quotient $\mathcal H\left(G, \  K, \ \mathcal A\right)/\mathfrak m_{x, \mathcal A}$ is  denoted by $F(x, \mathcal A)$. It is a finite (field)  extension of $\mathcal A$, and it also naturally
  irreducible $\mathcal A$--admissible $\mathcal H\left(G, \  K, \ \mathcal A\right)$--module. The map $Gal(\overline{\mathbb Q}/\mathcal A).x\longmapsto  F(x, \mathcal A)$  is a bijection 
  between $Gal(\overline{\mathbb Q}/\mathcal A)$--orbits in  $X(\overline{\mathbb Q})$, and the set of equivalence classes of irreducible $\mathcal A$--admissible irreducible
  $\mathcal H\left(G, \  K, \ \mathcal A\right)$--modules.

\item[(ii)] For each $x\in X(\overline{\mathbb Q})$, the $(\mathcal A, \  G)$--module (see Theorem \ref{eir-9} for the notation)
  $$
  \mathcal V(x, \mathcal A)\overset{def}{=}\mathcal V(\mathfrak m_x, K)
  $$
  is an irreducible and {\bf $\mathcal A$--admissible} $(\mathcal A, \ G)$--module. We have
  $$
  \mathcal V^K(x, \mathcal A)\simeq \mathcal B_{x, \mathcal A}
  $$ as $\mathcal H\left(G, \  K, \ \mathcal A\right)$--modules, 
  and
  $$
  \End_{(\mathcal A, \ G)} \left( \mathcal V(x, \mathcal A)\right)\simeq F(x, \mathcal A).
$$

\item[(iii)] $\mathcal V(x, \mathcal A)$ is absolutely irreducible  if and only if $x\in X(\mathcal A)$.

\item[(iv)] Let $x\in X(\overline{\mathbb Q})$. Then, for any Galois extension  $\mathcal A\subset \mathcal B$ which contains $F(x, \mathcal A)$,
  $\mathcal V(x, \mathcal B)$ is absolutely irreducible. Moreover, there exist $t=\dim_{\mathcal A} \ F(x, \mathcal A)$ 
  mutually different elements (among them $x$) in $Gal(\overline{\mathbb Q}/\mathcal B).x$, say $x=y_1, y_2, \ldots, y_t$, 
  such that  we have the following:
  $$
  \left( \mathcal V(x, \mathcal A)\right)_{\mathcal B}= \mathcal B\otimes_{\mathcal A} \  \mathcal V(x, \mathcal A)\simeq \mathcal V(x, \mathcal B)
\oplus   \mathcal V(y_2, \mathcal B)
\oplus \cdots \oplus  \mathcal V(y_t, \mathcal B).
$$
Furthermore,   $V(x, \mathcal B),  \mathcal V(y_2, \mathcal B),  \ldots,  \mathcal V(y_t, \mathcal B)$ are mutually non--isomorphic $(\mathcal B, \ G)$--modules.

\item[(v)] (Classification of unramified admissible representations  subfields of $\overline{\mathbb Q}$) The map
  $$
  Gal(\overline{\mathbb Q}/\mathcal A).x\longmapsto  \mathcal V(x, \mathcal A)
  $$
  is a bijection between $Gal(\overline{\mathbb Q}/\mathcal A)$--orbits in  $X(\overline{\mathbb Q})$, and the set of equivalence classes of unramified $\mathcal A$--admissible
  irreducible  $(\mathcal A, \ G)$--modules. 
\end{itemize}
  \end{Thm}
\begin{proof} It is obvious that the algebraic closure of $\mathcal A$ is   $\overline{\mathbb Q}$. This means that we can apply Lemma \ref{acuir-4} to any affine $\mathcal A$--variety.
  We apply it to $X$ which has the structure of affine $\mathcal A$--variety by letting
  $$
  \mathcal A[X]=\mathcal A\otimes_{\mathbb Q} \ \mathbb Q[X^*(\hat{A})]^{W}=\mathcal A[X^*(\hat{A})]^{W}.
  $$
  We identify $ \mathcal H\left(G, \  K, \ \mathcal A\right)$ with $\mathcal A[X]$ via $\mathcal A$--algebras isomorphism given by Lemma \ref{cuir-1}. 
  
By Lemma \ref{acuir-4}, for each $x\in X(\overline{\mathbb Q})$, there exists a unique maximal ideal $\mathfrak m_{x, \mathcal A} \subset \mathcal  A[X]$ such that
  $\mathfrak m_{x, \mathcal A}$ is the kernel of $\mathcal A$--algebra homomorphism $\mathcal A[X]\longrightarrow \overline{\mathbb Q}$ given by the evaluation at $x$;
  the image is a finite (field) extension, denoted by $F(x, \mathcal A)$  of $\mathcal A$. Two points in $X(\overline{\mathbb Q})$
  gives the same maximal ideal in $\subset \mathcal  A[X]$  if and only if they are  $Gal(\overline{\mathbb Q}/Q)$--conjugate.
  Now, (i) easily follows.

  In (ii), we use explicit  construction of $ \mathcal V(x, \mathcal A)\overset{def}{=}\mathcal V(\mathfrak m_x, K)$ from Theorem \ref{eir-9}. The isomorphism
  $\mathcal V^K(x, \mathcal A)\simeq \mathcal B_{x, \mathcal A}$ as $\mathcal H\left(G, \  K, \ \mathcal A\right)$--modules also follows from Theorem \ref{eir-9}.
  The deep thing is the fact that $ \mathcal V(x, \mathcal A)$ is $\mathcal A$--admissible. This is a consequence of our assumption that $\mathcal A\subset \overline{\mathbb Q}\subset \mathbb C$
  and Corollary \ref{ebf-0} to   Theorem \ref{aaeir-1}. Next, by Theorem \ref{aeir-1}, we have
 
  $$
  \End_{(\mathcal A, \ G)} \left( \mathcal V(x, \mathcal A)\right)\simeq  \End_{\mathcal H\left(G, \  K, \ \mathcal A\right)} \left( \mathcal V^K(x, \mathcal A)\right)
$$
But, since we have the following isomorphism of $\mathcal A$--algebras
$$
\mathcal H\left(G, \  K, \ \mathcal A\right)/\mathfrak m_{x, \mathcal A}\simeq F(x, \mathcal A),
$$
we have 
$$
\End_{\mathcal H\left(G, \  K, \ \mathcal A\right)} \left( \mathcal V^K(x, \mathcal A)\right)\simeq \End_{F(x, \mathcal A)} \left(F(x, \mathcal A)\right)=F(x, \mathcal A).
$$
This proves (ii).

(iii) follows from the characterization of absolutely irreducible modules given by Corollary \ref{eib-99}. Indeed, $\mathcal V(x, \mathcal A)$ is absolutely irreducible if and only if
$$
\End_{(\mathcal A, \ G)} \left( \mathcal V(x, \mathcal A)\right)\simeq \mathcal A.
$$
By (ii), we must have
$$
F(x, \mathcal A)= \mathcal A.
$$
Using the notation from the beginning of the proof, we have
$$
\mathcal A[x]/\mathfrak m_{x, \mathcal A}=F(x, \mathcal A)= \mathcal A.
$$
This equivalent to $x\in X(\mathcal A)$ by general theory of affine $\mathcal A$--varieties. This proves (iii). (v) follows from (i), (ii), and Lemma \ref{eir-6} (ii). 

Finally, we prove (iv). By Theorem \ref{aaeir-1}, there exists 
  irreducible  $(\mathcal B, \ G)$--modules
  $V_1, \ldots, V_t$ such that the following holds:
  \begin{itemize}
  \item[(i)] $V^K_i\neq 0$ for all  $1\le i\le t$. 
  \item[(ii)] $V^K_i$ are  $\mathcal B$--admissible  irreducible $\mathcal H\left(G, \ K, \ \mathcal B\right)$--modules.
  \item[(iii)] $V_{\mathcal B}\overset{def}{=}\mathcal B \otimes_{\mathcal A} V \simeq V_1\oplus \cdots \oplus V_t$ as $(\mathcal B, \ G)$--modules.
  \end{itemize} 
In order to identify modules $V_i$, an  argument from the proof of Lemma \ref{acuir-4} regarding tensor product of fields implies

$$
\mathcal B\otimes_{\mathcal A}  F(x, \mathcal A)= \mathcal B \oplus \cdots \oplus \mathcal B, \ \ \text{($\dim_{\mathcal A} \ F(x, \mathcal A)$) copies.}
$$
This can be considered as a decomposition of
$$
\mathcal H\left(G, \  K, \ \mathcal B\right)=\mathcal B\otimes_{\mathcal A} \mathcal H\left(G, \  K, \ \mathcal A\right)
$$
into irreducible modules. This implies that
$$
t=\dim_{\mathcal A} \ F(x, \mathcal A)
$$
in (iii) above. 

Since we have
$$
\mathcal B[X]= \mathcal B \otimes_{\mathcal A} \mathcal A[X],
$$
and obviously
$$
\mathfrak m_{x, \mathcal A}\mathcal B\subset \mathfrak m_{x, \mathcal B},
$$
we see that evaluation at $x$ for $\mathcal B$ i.e., 
$\mathcal B[X]\longrightarrow  F(x, \mathcal B)$ must come from an epimorphism
$$
 \mathcal B \oplus \cdots \oplus \mathcal B \longrightarrow 
 F(x, \mathcal A).
 $$
 Hence
 $$
 F(x, \mathcal A)=\mathcal B.
 $$
 This means that $x\in X(\mathcal B)$. In particular,  $\mathcal V(x, \mathcal B)$ is absolutely irreducible by (iii).
 Moreover,  each of $t=\dim_{\mathcal A} \ F(x, \mathcal A)$ different projections
 $\mathcal B \oplus \cdots \oplus \mathcal B \longrightarrow \mathcal B$ give rise 
 to the same number of different  epimorphisms of $\mathcal B$--agebras $\mathcal B[X]\longrightarrow \mathcal B$  that factor through $\mathfrak m_{x, \mathcal A}\mathcal B$.
 This means that they must correspond to evaluations at mutually different

 $$
 y_1, \ldots, y_t \in X(\overline{\mathbb Q})
 $$
 which belongs to $V(\mathfrak m_{x, \mathcal A})$ (see Lemma \ref{acuir-4} for the notation). One of them is $x$ as we proved above.
 Hence, they  must be mutually different elements (including $x$) in $ Gal(\overline{\mathbb Q}/\mathcal B).x$ by Lemma \ref{acuir-4} (iii).
 Now, (iv) follows.  We remark that $V(x, \mathcal B),  \mathcal V(y_2, \mathcal B),  \ldots,  \mathcal V(y_t, \mathcal B)$ are mutually non--isomorphic $(\mathcal B, \ G)$--modules.
 Since all $x=y_1, y_2, \ldots, y_t\in X(\mathcal B)$ because of the evaluation at them give $\mathcal B$ as an image.   Then, $\gamma.y_i=y_i$, for all $\gamma\in Gal(\overline{\mathbb Q}/\mathcal B)$, and
 $i=1, \ldots, t$. Now, we apply (v).
\end{proof}

\section{Appendix: A Result On Affine Varieties}\label{acuir}
We prove a simple general lemma which is an exercise for the exposition in (\cite{lang}, IX, Section 1).

\begin{Lem}\label{acuir-4} Let $k$ be a field of characteristic zero. We fix an algebraic closure $\overline{k}$ of $k$. Assume that $Z$ is (not necessarily  irreducible) affine variety
  over $k$. Then, we have the following:
  \begin{itemize}
    \item[(i)] If  $z\in Z(\overline{k})$ is any point, then $k[z]$ is a field. Therefore the kernel  of the $k$--homomorphism $k[Z]\longrightarrow k[z]$ is a maximal ideal, say $\mathfrak m$.
      We have $z\in V(\mathfrak m)$.
    \item[(ii)] Conversely,  let $\mathfrak m \subset k[Z]$ be a maximal ideal. Then, the variety $V(\mathfrak m)$, given as a common set of zeroes of $\mathfrak m$ in $Z(\overline{k})$, has a
      finite number of points. For each $z\in V(\mathfrak m)$, $k$--algebra $k[z]\subset \overline{k}$ is a finite extension of $k$. The evaluation at $z$ gives as isomorphism
      $k[Z]/\mathfrak m\simeq k[z]$ over $k$.
    \item[(iii)] Let $\mathfrak m \subset k[Z]$ be a maximal ideal. The variety $V(\mathfrak m)$ is defined over $k$.  $V(\mathfrak m)$ is a single $Gal(\overline{k}/k)$--orbit.
      The set of $k$--points $V(\mathfrak m)(k)$ of $V(\mathfrak m)$ is not empty if and only if $\mathfrak m$ is the kernel of (a unique)
        evaluation at $z\in Z(k)$. If this is so, $V(\mathfrak m)=\{z\}$.       
      \item[(iv)] $Z(\overline{k})$ is a disjoint union of all $V(\mathfrak m)$, where $\mathfrak m$ ranges over all maximal ideals of $k[Z]$.
   \end{itemize}   
  \end{Lem}
\begin{proof} We start with the following observation. The algebra $k[Z]$ is finitely generated $k$--algebra, say $f_1, \ldots, f_t$ are generators.
  Let $z\in Z(\overline{k})$. Then $f(z_i)\in \overline{k}$ and consequently $k[f(z_i)]$ is finite (field) extension of $k$. $k[z]$ is by definition $k[f(z_1), f(z_2), \ldots, f(z_t)]$ and it is a field
  and finite extension of $k$ by elementary field theory. This implies (ii).

  Let us prove (ii). We consider the ideal $I\subset \overline{k}[Z]$ defined by $I=\mathfrak m \cdot \overline{k}[Z]$. Then, obviously,
  $$
  Z(I)=Z(\mathfrak m).
  $$
  Now, by (\cite{lang}, IX, Section 1, Theorem 1.5), we have 
  $$
  Z(\mathfrak m)\neq \emptyset.
  $$
  Hence, $I$ is proper ideal. Also, there exists $z\in Z(\mathfrak m)$. For such $z$, $k[Z]/\mathfrak m\simeq k[z]$ is a finite extension of $k$. Let us put
  $$
  F=k[Z]/\mathfrak m.
  $$
  Then, since $k\subset F$ is finite and separabe extension (because $k$ has characteristic zero), there exists $\alpha \in F$ such that
  $$
  F=k(\alpha).
  $$
  Let $P\in k[T]$ be a minimal polynomial of $\alpha$, where $T$ is a variable. Then, let
  $$
  \alpha_1, \alpha_2, \ldots, \alpha_u, \ \ u=\deg{(P)},
  $$ 
  be all zeroes of $P$ in $\overline{k}$. They are all distinct.  The reader my easily check that
  \begin{equation} \label{acuir-5}
  \overline{k}\otimes_k  F\simeq \overline{k}\oplus \cdots \oplus \overline{k} \ \ \text{(a copy of $\overline{k}$ for each $\alpha_i$.)}
  \end{equation}
  Indeed, we have the following elementary and well--known computation:

  \begin{align*}
    \overline{k}\otimes_k  F &\simeq  \overline{k} \otimes_k  k[T]/k[T]P \\
    & \simeq   \overline{k}/ \overline{k}[T]P\\
    & =\overline{k}[T]/ \overline{k}[T](T-\alpha_1)(T-\alpha_2)\cdots (T-\alpha_u)\\
    &\simeq \oplus_{i=1}^u \ \overline{k}[T]/ \overline{k}[T](T-\alpha_i) \\
    &\simeq \oplus_{i=1}^u \ \overline{k}.
  \end{align*}

  We observe that (\ref{acuir-4}) implies 
  $$
  \overline{k}[Z]/I\simeq \overline{k}\otimes_k k[Z]/\overline{k}\otimes_k \mathfrak m \simeq \overline{k}\otimes_k \left(k[Z]/\mathfrak m\right)
  \simeq   \oplus_{i=1}^u \ \overline{k}.
  $$
  This shows that $I$ is a radical ideal since the right--hand side has no nilpotent elements. Hence,  $\overline{k}[Z]/I$ is algebra of regular functions on $V(\mathfrak m)$.
  And,  most importantly, $V(\mathfrak m)$ is a finite non--empty set. The rest of (ii) is clear. Next, (iv) is obvious from (i) and (iii). Finally, we prove (iii). It is well--known
  that $V\overset{def}{=}V(\mathfrak m)$ is defined over. Indeed, this follow from above considerations also. We have shown $\overline{k}[V]= \overline{k}[Z]/I$. If we let, 
  $k[V]=k[Z]/\mathfrak m$. Then, above isomorphism can be restated $\overline{k}[V]\simeq \overline{k}\otimes_k k[V]$,
  and it gives the  $k$--structure on $V$.

  To complete the proof of (iii),  we observe that $V(\mathfrak m)$ is a single $Gal(\overline{k}/k)$--orbit. Indeed,  let $z\in V=V(\mathfrak m)$. Then, for 
  $\gamma\in Gal(\overline{k}/k)$, we have $\gamma.z\in V(\mathfrak m)$ since by the definition of the Galois action on $Z$: 
  $$
  f(\gamma.z)=\gamma^{-1}(f(z))=\gamma(0)=0.
  $$
  Conversely, if $z$ and $z'$ are in $V$. Then, the fields $k[z]$ and $k[z']$ are isomorphic to $k[V]$ over $k$. Thus, there exists $\gamma\in  Gal(\overline{k}/k)$
  such that $\gamma(k[z])=k[z']$. Equivalently, $k[\gamma.z']=k[z]$.  This means that
  $$
  f(\gamma.z')=f(z), \ \ \text{for all $f\in k[Z]$.}
  $$
  Hence, we have 
  $$
  f(\gamma.z')=f(z), \ \ \text{for all $f\in \overline{k}[Z]$.}
  $$
  This means that
  $$
  \gamma.z'=z.
  $$
  The rest of (iii) is clear. 
  \end{proof}

\end{document}